\documentclass[12pt,letterpaper,reqno]{amsart}
\usepackage[letterpaper,margin=1.2in,headheight=15pt]{geometry} 
\usepackage{graphicx,bbm}
\usepackage{amsmath, amssymb, amsfonts}
\usepackage[utf8]{inputenc}
\usepackage{epstopdf}
\usepackage{tikz-cd} 
\usepackage{setspace}
\usepackage{xcolor}
\usepackage{tcolorbox}
\usepackage[titletoc,title]{appendix}
\usepackage{enumitem}

\usepackage{hyperref}
\definecolor{darkred}{rgb}{0.5,0.15,0.15}
\hypersetup{colorlinks=true,urlcolor=darkred,linkcolor=darkred,citecolor=darkred}

\usepackage{mathtools}

\usepackage{cleveref} 

\newcommand{\ii}{\textnormal{i}}

\renewcommand{\ell}{X} 

\newcommand{\I}{{\mathrm i}}

\newcommand{\be}{\begin{eqnarray}}
\newcommand{\ee}{\end{eqnarray}}
\newcommand{\bea}{\begin{eqnarray}}
\newcommand{\eea}{\end{eqnarray}}

\newcommand{\ben}{\begin{eqnarray}}
\newcommand{\een}{\end{eqnarray}}

\theoremstyle{plain}
\newtheorem{thm}{Theorem}[section]
\newtheorem{prop}{Proposition}[section]

\newtheorem{introtheorem}{Theorem}

\theoremstyle{definition}
\newtheorem{dfn}{Definition}[section]

\newtheorem*{question*}{Question}

\theoremstyle{remark}
\newtheorem{rem}{Remark}[section]

\numberwithin{equation}{section}

\title[From Closed to Relative GW via Resurgent Functions]{From Closed to Relative Higher-Genus Gromov--Witten Invariants via Resurgent Functions    }

\author{Murad Alim}
\author{Noah Tischler}
\address{Maxwell Institute for Mathematical Sciences, Edinburgh EH14 4AS, UK\\
Department of Mathematics, Heriot-Watt University, Edinburgh EH14 4AP, UK\\
\small Department of Mathematics, Technical University of Munich, Boltzmannstr. 3, 85748, Garching, Germany}

\begin{document}

\begin{abstract}

Higher genus Gromov--Witten invariants of Calabi--Yau threefolds are encoded in a generating function which is an asymptotic series in a formal parameter $\lambda$. Using resurgence, analytic functions in this formal parameter were uncovered. 
In this paper we focus on the resolved conifold and study the enumerative meaning of the strong-coupling asymptotic expansion, in powers of $ 1/\lambda $, of the resurgent analytic functions. We show that this expansion contains both a closed curve-counting contribution in dual variables and a contribution governed by relative Gromov--Witten invariants, naturally interpreted in logarithmic geometry.

\end{abstract}
\maketitle

\tableofcontents
\section{Introduction}

Modern enumerative geometry and especially Gromov--Witten (GW) theory  has enormously benefited from interactions with physics. Using mirror symmetry of Calabi--Yau (CY) threefolds, which was conjectured within String Theory, the generating function of all degree rational curves on the quintic was obtained from a variation of Hodge structure computation on the mirror quintic \cite{Candelas:1990rm}. This triggered  broad interest within mathematics leading to mathematical proofs of these results \cite{LLY,GiventalMirror}. 

Mirror symmetry of CY threefolds matches flat connections constructed for mirror families of CY threefolds. On the B-side, the flat connection arises from the variation of Hodge structures of the middle-dimensional cohomology of a threefold $\check{X}$ over the moduli space of complex structures. On the A-side, which has direct relations to Gromov--Witten theory of the mirror CY $X$, the corresponding flat connection is provided by the quantum cohomology of the even part of the cohomology. In addition to being a powerful computational tool, this match of flat connections via mirror symmetry reveals a global structure of the generating functions of rational invariants, different regions in the underlying moduli space of complexified K\"ahler structures relate GW theories of topologically different CY manifolds and lead to orbifold GW invariants. See \cite{coxkatz,Hori:2003ic,ChenRuan,AbramovichLectures} and references therein for an exposition on these topics.

The generating function of higher genus GW invariants of a CY $X$ in turn is captured by the partition function of topological string theory. A recursive structure of the higher genus generating functions is provided by the holomorphic anomaly equations of the mirror CY $\check{X}$  \cite{Bershadsky:1993cx}, which give the higher genus generating functions an interpretation within geometric quantization \cite{Witten:1993ed}. 

The topological string partition function is given as an expansion in a formal parameter $\lambda$, the topological string coupling, summing over genera of Riemann surfaces. It has the form:
\begin{equation}
    Z_{top}(\lambda,t)= \exp \left( \sum_{g=0}^{\infty} \lambda^{2g-2} F^g(t)\right)\,,
\end{equation}
where $t\in \mathcal{M}$ denotes a set of local coordinates on the underlying moduli space of the CY family. The global structure on $\mathcal{M}$ of each $F^g(t)$ can be accessed through a Feynman diagram expansion presented in \cite{Bershadsky:1993cx}. An underlying polynomial structure expressing higher genus GW generating functions in terms of finitely many generators at genus zero was proven in \cite{Yamaguchi:2004bt} for the mirror quintic and generalized in \cite{Alim:2007qj} for arbitrary families of CY threefolds. This polynomial structure significantly enhances the computability of higher genus $F^g$, see e.g. \cite{Huang:2006hq} and allows one to access the global structure of the $F^{g}$'s on the moduli space $\mathcal{M}$, see e.g. \cite{Haghighat:2008gw,Alim:2008kp}. 

However, a global structure in $\lambda$ is not clear from the outset. Within the physical formulation $\lambda$ is considered a formal small expansion parameter, the recursive structure leads to an asymptotic series in $\lambda$, within the geometric quantization interpretation of \cite{Witten:1993ed}, this parameter is put in relation to the quantization parameter $\hbar$, which generically does not have a global structure. The question regarding the global structure in $\lambda$ is thus intimately related in the physical setup to a non-perturbative formulation of the theory. Within topological string theory, this question has been addressed in many works, see \cite{Marino:2024tbx} and references therein. Within enumerative geometry, clear indications of the relevance of a global structure in $\lambda$ have already been put forward in the works of Maulik, Nekrasov, Okounkov and Pandharipande \cite{MNOP1,MNOP2}, where a change of variables
$q=\exp(\ii \lambda)\,$
relates Gromov--Witten generating functions with certain counts of Donaldson--Thomas (DT) invariants on the same CY threefold.

In the present work, we focus on one particular promising approach to the global structure of higher genus Gromov--Witten theory, namely resurgence developed by \'Ecalle \cite{ecalle1981resurgent}; see also \cite{SauzinLectures} and references therein for an overview. Within resurgence one can construct systematically piecewise analytic functions from a given divergent series. The goal of our work is to study the enumerative geometric content of the global analytic structure associated to the formal higher genus counting parameter $\lambda$, focusing on the geometry of the resolved conifold where this question can be addressed and answered explicitly. In particular, we use the results of \cite{Alim:2021mhp}, where piecewise analytic functions were constructed using resurgence for the Gromov--Witten potential of the resolved conifold, following methods which were employed in the study of the quantum dilogarithm \cite{Garoufalidis:2020pax}. See also \cite{Pasquetti:2009jg,Hatsuda_2014} for earlier studies using resurgence for this geometry, and \cite{Grassi:2022zuk,2023Alim,alim2024nonperturbativetopologicalstringsresurgence} for related work. 

In \cite{Alim:2021mhp}, the authors gave a precise relation of the analytic functions and their Stokes discontinuities (jumps) to wall-crossing structures in Donaldson--Thomas theory via the relation to Bridgeland's Riemann-Hilbert problem in that context \cite{BridgelandCon}. Our goal here is to address another question, namely a dual asymptotic expansion of the aforementioned analytic functions in terms of $\lambda_D= \frac{4\pi^2}{\lambda}$. The Borel--Laplace resummation employed within resurgence gives an analytic function in $\lambda$ for different choices of half-planes in the complex $\lambda$ plane. A particular \emph{maximal} one gives an analytic function for the half-plane $\Re \lambda >0$, we will denote this by $F_{np}(\lambda,t)$ for the GW generating function, and by $Z_{np}(\lambda,t)=\exp F_{np}(\lambda,t)$ for the partition function. The latter has a product expansion which reads for the resolved conifold \cite{alim2024nonperturbativetopologicalstringsresurgence}:

\begin{equation}
Z^{con}_{np}(\lambda,t):= \prod_{m=0}^{\infty} (1-Q\, q^{m+1})^{m+1}  \cdot \exp\left(-\frac{1}{2\pi i} \textrm{Li}_2(Q' q'^m)\right) \cdot (1-Q' q'^m)^{-\frac{1}{\check{\lambda}}\left( t+m\right)}\,,
\end{equation}
where:
\begin{equation}
Q=\exp(2\pi i t),\quad q=e^{\ii\lambda}\,, \quad Q'=\exp(2\pi i t/\check{\lambda})\,, \quad q'=\exp(2\pi i/\check{\lambda}),\quad \check{\lambda}=\frac{\lambda}{2\pi}\,.
\end{equation}

The first part of this product 
\begin{equation}
Z_{GV}(\lambda,t)= \prod_{m=0}^{\infty} (1-Q\, q^{m+1})^{m+1} \,,
\end{equation}

corresponds to the relation between GW and DT put forward by MNOP \cite{MNOP1,MNOP2}. The parts of the product containing $Q'$ and $q'$ are however only accessible through the analytic structure in $\lambda$. These parts have been identified with a limit of the refined topological string partition function, namely the Nekrasov--Shatashvili \cite{NEKRASOV_2010} limit, guided by a general structure appearing in \cite{Hatsuda:2015owa,Grassi:2014zfa}.

In terms of the variables  \eqref{eq:dualvariables}
\begin{equation*} 
\lambda_D= \frac{4\pi^2}{\lambda},\, t_D=\frac{2\pi t}{\lambda} \quad \textrm{and} \quad \check{\lambda}_D= \frac{\lambda_D}{2\pi}=\frac{1}{\check{\lambda}}
\end{equation*} 
$F_{np}^{con}(\lambda,t )= \log Z_{np}^{con}(\lambda,t)$ can be written as \eqref{GVNS}
 \begin{equation*}
\widetilde{F}^{con}_{\text{np}}(\lambda,t)= F_{\text{GV}}(\lambda,t) + \frac{1}{2\pi} \frac{\partial}{\partial \lambda} \left(\lambda \, F_{\text{NS}} \left(\lambda_D,t_D - \frac{1}{2}\check{\lambda}_D\right)\right)\,,
\end{equation*}
where the second part can furthermore be written as a sum of three contributions \eqref{derivativeFNS}:

\begin{equation*}
\begin{split}
\frac{1}{2\pi} \frac{\partial}{\partial \lambda} \left(\lambda \, F_{\text{NS}} \left(\lambda_D,t_D - \frac{1}{2}\check{\lambda}_D\right)\right)
&= \frac{1}{2\pi} F_{\text{NS}} \left(\lambda_D,t_D - \frac{1}{2} \check{\lambda}_D\right) \\
&+ \frac{1}{2\pi} \lambda_D F_{GV} (\lambda_D,t_D) \\ &-\frac{1}{2\pi} t_D \partial_{t_D}F_{\text{NS}} \left(\lambda_D,t_D - \frac{1}{2}\check{\lambda}_D\right)\,.
\end{split}
\end{equation*}

Our work thus addresses the enumerative geometry and Gromov--Witten curve count of these three pieces of the potential. The GW content of this part is thus intimately related to the GW content of the NS limit of refined topological strings.

Recent work \cite{bousseau2020proofntakahashisconjecturemathbbp2e,Bousseau_2021, gräfnitz2025enumerativegeometryquantumperiods, brini2024refinedgromovwitteninvariants} studied the enumerative properties of the NS limit of the refined topological string on local Calabi--Yau threefolds and related it to log/relative curve counts on log Calabi--Yau pairs (or in a similar fashion via equivariant GW theory on the Calabi--Yau fivefold $\mathsf{Tot}_{S}\left( K_S \oplus \mathcal O_S \oplus \mathcal O_S\right)$ for $S$ a smooth toric del Pezzo).\par
The geometry of the NS limit of the refined topological string on a CY threefold became of interest in trying to understand more general curve/sheaf counting correspondences, predicted from the physics literature, an example being the prediction of quasi-modularity of $F_{\text{NS}}^g$ worked out in \cite{huang2010directintegrationgeneralomega} and proven for $\mathsf{Tot}_{\mathbb P^2}\left(\mathcal O(-3) \right)$ in \cite{Bousseau_2021}. Similarly, a closely related quantity to the NS limit free energy, the quantum $\mathsf A$-period, which one can think of as a quantum corrected mirror map, with refinement parameter $q\to 1$ recovering the usual mirror map, was studied in depth in \cite{gräfnitz2025enumerativegeometryquantumperiods}, where an enumerative meaning in terms of tropical curve counts was determined. The above results are established in the framework of logarithmic (log) geometry.\par
Based on the work \cite{Kato1989LogStructures}, which originates from ideas by various mathematicians including Deligne, Faltings, Illusie and Kato, see \cite{logintroduction} for a historical account, log geometry was discovered as a tool to make sense of geometric questions in mildly singular spaces. The interest in log GW invariants finds its origins in the Gross-Siebert program (see \cite{gross2006mirror, gross2010mirror, gross2013logarithmic,gross2015mirrorsymmetrylogcalabiyau,gross2016intrinsic} as well as \cite{chen2010logarithmic, abramovich2010logarithmic} and references therein), an algebraic geometric incarnation of mirror symmetry.

It was then realized \cite{abramovich2013comparisontheoremsgromovwitteninvariants} that log curve counts are the objects unifying the various notions of relative GW theories established at the time, see \cite{li1998symplectic, li2001stable,ionel2003relative}. The idea is that in Gromov--Witten theory when working relative to a divisor i.e. imposing tangency conditions, the log geometric description gives a natural compactification of the moduli space of stable maps with specified contact profile to the divisor. 

In this work we are interested in studying the enumerative properties of the non-perturbative free energy $\widetilde{F}_{\text{np}}(\lambda, t)$ and therefore the NS limit of the free energy of the resolved conifold in particular. We then establish the relationship between relative curve counts on $\mathbb P^1$ with maximal tangency to a divisor $D\subset \mathbb P^1$ and the NS contribution to the large $\lambda$ regime of the analytic function $\widetilde{F}_{\text{np}}(\lambda, t)$ obtained by resurgent analysis on the closed GW generating function.

The global analytic object $\widetilde{F}_{\text{np}}(\lambda, t)$ containing both closed curve counting invariants as well as contributions from the NS limit which we will identify with relative GW invariants in appropriate asymptotic regimes seems to suggest an interplay between two a priori very different curve counting theories: closed- and relative curve counts. \par
It is this global phenomenon which we examine in this work for the resolved conifold. The main goal is to give an enumerative interpretation of the strong coupling asymptotics of $\widetilde{F}^{con}_{\mathrm{np}}(\lambda, t)$ in terms of log/relative GW invariants of $\mathbb P^1 \setminus \infty$.

\subsection{Relation to other works}
We proceed with an overview of works that relate to our setup. The main idea of giving an enumerative interpretation of the strong coupling asymptotics of $\widetilde{F}^{con}_{\mathrm{np}}(\lambda, t)$ was inspired by the recent developments of the understanding of the NS limit in terms of logarithmic/relative invariants for log pairs consisting of a surface and an anticanonical divisor.
\\

In the work \cite{van2019local}  the authors prove the relation between local- and relative/logarithmic invariants in genus zero. This relates to our work in the following way. Namely our result from Theorem \ref{theorem A} is a $q$-deformation of the main theorem \cite[Theorem 1.1]{van2019local}, where the relationship between the two (virtual) enumerative invariants is proven at the level of virtual fundamental cycles.
\\

Furthermore, the relationship between local- and log- and open invariants in higher genus was also studied in \cite{bousseau2024stable}, giving an intricate network of relationships between a priori very different curve- and sheaf counting theories. Specifically, the relationship between open- and log invariants is explored, which is important for our work. We study an Aganagic--Vafa brane with framing zero on an outer toric leg. The machinery developed in \cite{bousseau2024stable} relates invariants of our setup to maximal tangency relative GW invariants of the Looijenga pair $(\mathbb F_1, D=D_1+D_2)$, where $D= H + (2 H- E)$. Here $H$ is the pullback of the hyperplane class along $\pi : \mathbb F_1 \cong \mathsf{Bl}_{p}\mathbb P^2 \to \mathbb P^2$ and $E$ is the exceptional divisor of the blow-up at the torus fixed point $p\subset \mathbb P^2$. Then these invariants in the NS limit would correspond to not iterated $q$-log derivatives but log derivatives w.r.t. the K\"ahler parameters of $D$.
\\

In \cite{bousseau2020proofntakahashisconjecturemathbbp2e} Bousseau establishes that refined sheaf counts on $\mathbb P^2$ are related to relative GW invariants relative to a smooth, anticanonical divisor after a change of variable. The refined sheaf invariants are defined as
\begin{equation}
    \Omega^{\mathbb P^2}_d \left( q^{1/2}\right) \coloneqq (-q^{1/2})^{-(d^2+1)} \sum_{j=0}^{d^2+1}\mathsf{I b}_{2j}\left( M_{d,\chi}\right) q^j \in \mathbb Z[ q^{\pm 1/2}]
\end{equation}
where\footnote{In \cite{Maulik_2023} it is proven that the invariants $ \Omega_d \left( q^{1/2}\right)$ are independent of the numerical invariant $\chi$.} $M_{d,\chi}$ is the $\dim_{\mathbb C}M_{d,\chi} = d^2 +1$ dimensional Le Potier moduli space of Gieseker semistable $\mathcal F$ sheaves on $\mathbb P^2$ at fixed numerical invariants $d \coloneqq \deg \left(\mathsf{supp}(\mathcal F) \right)$ and $\chi \coloneqq \chi(\mathcal F)$ and by $\mathsf{I b}_{2j}$ we mean the intersection Betti numbers (for generic $d, \chi$ $M_{d, \chi}$ is singular). Furthermore in \cite[Theorem 1.12]{bousseau2020proofntakahashisconjecturemathbbp2e} Bousseau establishes that one can express the NS limit of the refined topological string free energy on local $\mathbb P^2$ (i.e. $\mathsf{Tot}_{\mathbb P^2}\left( \mathcal O(-3)\right)$ which is evidently CY) in terms of the relative GW invariants
\begin{equation}
    F_{NS}^{loc.\mathbb P^2}(\epsilon, Q) \coloneqq  \sum_{g \in \mathbb Z_{\geq 0}} \sum_{d \in \mathbb Z_{\geq 1}} \frac{(-1)^{d-1}}{3d}N_{g,d}^{\mathbb P^2 \setminus E} Q^d \epsilon^{2g-1}
\end{equation}
with maximal contact order $d$ to a smooth elliptic curve on $\mathbb P^2$. This may be viewed as a mathematically rigorous definition of the NS free energy in the local del Pezzo surface setting.
\par
Then there is a correspondence between $ \Omega^{\mathbb P^2}_d \left( q^{1/2}\right)$ and  $F_{NS}^{loc.\mathbb P^2}(\epsilon, Q) $ given by
\begin{equation}
   \sum_{g \in \mathbb Z_{\geq 0}} \sum_{d \in \mathbb Z_{\geq 1}} \frac{(-1)^{d-1}}{3d}N_{g,d}^{\mathbb P^2 \setminus E} Q^d \epsilon^{2g-1} \overset{q = \exp(-i \epsilon)}{=}\\
   i \sum_{d \in \mathbb Z_{\geq 1}} \sum_{l \in \mathbb Z_{\geq 1}} \frac{1}{l^2} \frac{\Omega_d^{\mathbb P^2} \left( q^{1/2}\right)}{q^{l/2}- q^{-l/2}} Q^{l d}.
\end{equation}
\\

This correspondence between refined sheaf counting and GW invariants is also explored further in \cite{brini2024refinedgromovwitteninvariants} from the angle of equivariant GW theory of certain CY fivefolds.
There, Brini and Sch\"uler establish \cite[Theorem A1]{brini2024refinedgromovwitteninvariants}:
Let $T$ be a torus acting antidiagonally on the fibers of each $\mathcal O(-1) \oplus \mathcal O$ summand. Then the $T$-equivariant GW invariants of the resolved conifold are
\begin{equation}
    GW_{d}\left( \mathsf{Tot}_{\mathbb P^1}( \mathcal O(-1)^{\oplus 2} \oplus \mathcal O^{\oplus 2}), T\right) = \frac{1}{d}\frac{1}{(q_1^{d/2}- q_1^{-d/2})(q_2^{d/2} - q_2^{-d/2})}
\end{equation}
where $q_i = \exp(-\sqrt{-1} \epsilon_i)$ for $i \in \{1,2\}$. From this, one obtains
\begin{equation}
    \Omega_{d}^{\mathbb P^1}(q^{1/2}) = \delta_{d,1}
\end{equation}
where $\Omega_{d}^{\mathbb P^1}(q^{1/2})$ denotes the Poincar\'e polynomial of the moduli space of Gieseker semistable sheaves on $\mathbb P^1$ (with Euler characteristic one).
\\

Lastly we note that the generating function of the relative GW invariants of $\mathbb P^1 \setminus \infty$ was first computed in \cite{faber2003relativemapstautologicalclasses}. Here we use the technology developed by Faber-Pandharipande \cite{Faber, faber2003relativemapstautologicalclasses} as well as Bryan-Pandharipande \cite{bryan2004curvescalabiyau3foldstopological, bryan2006localgromovwittentheorycurves} to obtain the relative invariants:
\begin{equation*}
    N^{\mathbb P^1 \setminus \infty}_{g,d} =(-1)^{d-1}d^{2g-2} b_g = \begin{cases}
    \frac{(-1)^{d-1}}{d^2},& g=0\\
        (-1)^{d-1} d^{2g-2} \frac{2^{2g-1} -1}{2^{2g-1}} \frac{\vert B_{2g}\vert}{(2g)!}, & g \in \mathbb Z_{\geq 1}
        \end{cases}
\end{equation*}
where we note that we define in Section \ref{sec:relative gw invariants} the invariants $N^{\mathbb P^1 \setminus \infty}_{g,d}$ as in \eqref{relativeGW computation}, so with insertion of the obstruction bundle
\begin{equation}
    V \coloneqq R^1 \pi_{\ast} \left( \mathcal O_{\mathcal C}\oplus \mathsf f^\ast \mathcal O_{\mathbb P^1}(-1)\right)\,
\end{equation}
where $\pi: \mathcal C \to \overline{\mathcal M}^{\dagger}_{g,n}\left( \mu^1\right)$ is the respective universal curve with $\mathsf f:\mathcal C \to \mathbb P^1$.

\subsection{Results}
Our aim in this work is to explore a similar description of the NS limit and its enumerative interpretation for the resolved conifold for two reasons:
\begin{enumerate}
    \item The technology used in \cite{bousseau2020proofntakahashisconjecturemathbbp2e} for $\mathbb P^2 \setminus E$ is not (immediately) applicable to the resolved conifold. In principle, the resolved conifold is substantially simpler to treat than the local surface CYs. While \cite{brini2024refinedgromovwitteninvariants} do tackle the problem of local $\mathbb P^1$ geometries with full refinement, i.e. $ \varepsilon_1 $ and $ \varepsilon_2 $ are both nonzero and not specialized to the unrefined or NS limits, we come from different considerations. We are concerned with elaborating on the enumerative geometry of the non-perturbative free energy of the resolved conifold, of which a constituent is the NS free energy. 
    \item While the NS limit is  known up to all orders by work of \cite{Aganagic_2012}, in which they compute the free energy by identifying a spectral curve associated to the resolved conifold, quantizing this curve and using the WKB method to compute the generating function genus by genus recursively, we are interested in the apparent duality suggested by resurgence: We obtain the NS generating function (given in terms of relative GW invariants) purely from considerations about the closed GW generating function. We then give explicit computations of the NS free energy as well as the derivation of the relative invariants and compare the generating functions.
\end{enumerate}
\par

Both the free energy in the Nekrasov--Shatashvili limit (\cite{Aganagic_2012}) as well as the Hodge integrals (\cite{faber2003relativemapstautologicalclasses}) have been previously computed. It is the identification of the latter
as the strong-coupling enumerative sector of the resurgent analytic completion constructed from only the resolved conifold’s topological string free energy that we study. Thus a single analytic function
interpolates, through either small- or large string-coupling asymptotic expansions, between closed and
relative curve-counting theories.

\begin{introtheorem}[= Theorem $4.1$]\label{theorem A}
    Fix the notation $Q_T=\exp(2\pi iT)$ and $q'=\exp(i\lambda_D)$. 
Then the Nekrasov--Shatashvili free energy of the resolved conifold satisfies
\begin{equation}
F_{\mathrm{NS}}(\lambda_D,T) = \sum_{g\geq 0} \sum_{d\geq 1} \frac{1}{d} N^{\mathbb{P}^1\setminus\infty}_{g,d} (-Q_T)^d \lambda_D^{2g-1}.
\end{equation}
In particular, evaluating at the shifted dual Kähler parameter $T=T_D$, the relative Gromov--Witten invariants of $\mathbb{P}^1$ with maximal tangency to the divisor $\infty$ give the Nekrasov--Shatashvili contribution appearing in the dual expansion of $\widetilde F^{con}_{\mathrm{np}}(\lambda,t)$.

Equivalently, after the GV-type resummation
one obtains
\begin{equation}
F^{GV}_{\mathrm{NS}}(q',T) = -i \sum_{k\geq 1}
\frac{Q_T^k}{k^2(q'^{k/2}-q'^{-k/2})}.
\end{equation}
Consequently, we obtain for the Poincar\'e polynomials of the moduli space of (degree $d$) Gieseker semistable sheaves on $\mathbb P^1$
\begin{equation}
\Omega_d^{\mathbb{P}^1}(q'^{1/2})=\delta_{d,1}.
\end{equation}
    
\end{introtheorem}

\subsection{Outline}
The paper is organized as follows. 
Our objective is to show an interplay of curve counting theories for the resolved conifold, which manifests itself in the form of different asymptotic behaviors of a global object, an analytic function which we call the non-perturbative free energy.

\underline{In Section \ref{sec:freeenergies}}, we will review the basic notions of closed GW theory for a CY threefold. We then specialize the discussion to the conifold. 

\underline{In Section \ref{sec:analytic structure}} we summarize the main results of the resurgence analysis of the closed GW invariants generating function, following \cite{Alim:2021mhp} and references therein. This gives us an analytic function in the string coupling $\widetilde{F}_{\text{np}}^{con}(\lambda, t)$ in particular on the distinguished half-plane $ \Re \lambda>0 $, with values depending on the Kähler parameter.

\underline{In Section \ref{sec:relative gw invariants}} our main objective is to compute the log/relative invariants for the resolved conifold. To do this, we will recall basic facts necessary to follow the localization computation done in \cite{bryan2004curvescalabiyau3foldstopological} in order to compute the relative invariants $N^{\mathbb P^1 \setminus \infty}_{g,d}$. To conclude this section, we then show how the previously computed invariants relate to the NS free energy and therefore the large coupling asymptotics of the analytic partition function.

\underline{In Section \ref{sec:outlook}}, we end on a summary of the ideas presented as well as further interesting directions which have not yet been studied systematically, especially beyond the resolved conifold geometry.

\subsection*{Acknowledgments}
It is a pleasure to thank Andrea Brini, Pierrick Bousseau and Valdo Tatitscheff for discussions related to this work as well as giving invaluable feedback on earlier versions of this draft. The work of NT is supported by the EPSRC Centre for Doctoral Training in Algebra, Geometry and Quantum Fields.

\section{Topological strings and enumerative geometry}\label{sec:freeenergies}
To a mirror family of CY threefolds, topological string theory associates the topological string partition function which is defined as an asymptotic series in the topological string coupling $\lambda$, summing over the free energies $ \mathcal{F}^{g}(t)$ associated to Riemann surfaces of genus $g$:
\begin{equation}
Z_{top} (\lambda,t)= \exp \left(\sum_{g=0}^{\infty} \lambda^{2g-2} \mathcal{F}^{g}(t)\right)\,,
\end{equation}
where $t=(t^1,\dots,t^n)$ is a set of distinguished local coordinates on the underlying moduli space $\mathcal{M}_{\text{K\"ahler}}$, which is of dimension $n=h^{1,1}(X_t)=h^{2,1}(\check{X}_{t(z)})$. $X_t$ and $\check{X}_{t(z)}$ are a mirror pair of CY threefolds which correspond to the A-model and B-model sides of mirror symmetry. 


\subsection{The Gromov--Witten potential}
In a certain limit together with an expansion around a distinguished large volume point in the moduli space $\mathcal{M}_{\text{K\"ahler}}$, the topological string free energies $\mathcal{F}^{g}(t)$ become the generating functions of higher genus Gromov--Witten invariants $F^g(t)$ on the A-model side of mirror symmetry. The GW potential of $X$ is the following formal power series:
\begin{equation}
F_{GW}(\lambda,t) = \sum_{g\ge 0}  \lambda^{2g-2} F^g(t)= \sum_{g\ge 0}  \lambda^{2g-2} \sum_{\beta\in H_2(X,\mathbb{Z})}  [\textrm{GW}]_{\beta,g} \,Q^{\beta}\, ,
\end{equation}
where $Q^{\beta} \coloneqq \exp (2\pi \I t^{\beta} )$ is a formal variable and $[\textrm{GW}]_{\beta,g}$ are the Gromov--Witten invariants associated to a curve of genus $g$ and class $\beta$. Let $C_a \in H_{2}(X,\mathbb{Z}), a=1,\dots,n=h^{1,1}(X)$ be a set of curve classes spanning $H_2(X,\mathbb{Z})$, then in general all curve classes $\beta$ can be written as $\beta=\sum_{a=1}^{n} d_a C_a$, where the integers $d_a$ are the degrees of the curve. The formal variables $t^{\beta}$ become $t^{\beta}=\sum_{a=1}^{n} d_a t^a$ and the $t^a$ correspond to local coordinates on the moduli space of complexified K\"ahler forms of the underlying CY $X$. Bearing this in mind, we will in the following treat $t^{\beta}$ as independent formal variables for the curve classes $\beta$ for ease of exposition of the general results.

The GW potential can furthermore be written as:
\begin{equation}
F_{GW}=F_c + \tilde{F}_{GW}\,,
\end{equation}
where $F_c$ denotes the contribution from constant maps and $ \tilde{F}_{GW}$ the contribution from non-constant maps. The constant-map contributions at genera 0 and 1 are $ t $-dependent and the higher genus constant map contributions take the universal form \cite{Faber}:
\begin{equation}
F_c^g\coloneqq \frac{\chi(X)(-1)^{g-1}\, B_{2g}\, B_{2g-2}}{4g (2g-2)\, (2g-2)!}\,, \quad g\ge2\,,
\end{equation}
where $\chi(X)$ is the Euler characteristic of $X$ and the Bernoulli numbers $B_n$ are generated by:
\begin{equation}
\frac{w}{e^w-1} = \sum_{n=0}^{\infty} B_n \frac{w^n}{n!}\,.
\end{equation}

The Gopakumar--Vafa (GV) resummation of the GW potential \cite{Gopakumar:1998ii,Gopakumar:1998jq} reformulates the non-constant part of the GW potential in terms of the Gopakumar--Vafa invariants  $[\textrm{GV}]_{\beta,g} \in \mathbb{Z}$ which are given by a count of electrically charged $\mathsf M2$ branes in an M-theory setup. The GW potential can thus be written as:
\begin{equation}\label{GVresum}
F_{GV}(\lambda,t)= \sum_{\beta>0}\sum_{g\ge 0} [\textrm{GV}]_{\beta,g}\, \sum_{k\ge 1} \frac{1}{k} \left( 2 \sin \left( \frac{k\lambda}{2}\right)\right)^{2g-2} Q^{k\beta}\, .
\end{equation}
In particular $F_{GV}$ has the same asymptotic expansion at $\lambda=0$ as $\tilde{F}_{GW}$
and 
$$ \tilde{F}_{GW}^{g=0}(t)=\sum_{\beta>0} [\textrm{GV}]_{\beta,0}\, \textrm{Li}_3(Q^{\beta})\,, \quad Q^{\beta}= \exp(2\pi i t^{\beta})\,.$$


\subsection{GW potential of the resolved conifold}
An example of a non-compact CY manifold $X$  for which the asymptotic series is known to all orders is the 
resolved conifold which is given by the total space of the rank-two bundle over the projective line:
\begin{equation}
X \coloneqq \mathcal{O}_{\mathbb P^1}(-1) \oplus \mathcal{O}_{\mathbb P^1}(-1) \rightarrow \mathbb{P}^1\,,
\end{equation}
and corresponds to the resolution of the conifold singularity.\\ 

The GW potential for this geometry was determined in physics \cite{Gopakumar:1998ii,GV}, and in mathematics \cite{Faber} with the following outcome for the non-constant maps:
\begin{equation}\label{resconfree}
\widetilde{F}^{con}(\lambda,t)= \sum_{g=0}^{\infty} \lambda^{2g-2} \widetilde{F}^g(t)= \frac{1}{\lambda^2} \mathrm{Li}_{3}(Q)+\frac{B_2}{2}\mathrm{Li}_1(Q)+ \sum_{g=2}^{\infty} \lambda^{2g-2} \frac{(-1)^{g-1}B_{2g}}{2g (2g-2)!} \mathrm{Li}_{3-2g} (Q) \, ,
\end{equation}
using the notation $Q=e^{2\pi\ii t}$. 
The GV resummation in this case gives:
\begin{equation}
F_{GV}(\lambda,t)=\sum_{k=1}^\infty\frac{e^{2\pi \I k t}}{k\big(2\sin\big(\frac{\lambda k}{2}\big)\big)^2}\,.
\end{equation}


\section{Analytic structure in the topological string coupling} \label{sec:analytic structure}

\subsection{The resolved conifold}\label{subsec:rescon}

The starting point is the formal series giving the non-constant and non-classical piece of the Gromov--Witten potential or the topological string free energy of the resolved conifold:
\begin{align*}\label{formal}
\widetilde{F}^{con}(\lambda,t)&=  \frac{1}{\lambda^2} \mathrm{Li}_{3}(Q)+\frac{B_2}{2}\mathrm{Li}_1(Q) + \sum_{g=2}^{\infty} \lambda^{2g-2} \frac{(-1)^{g-1}B_{2g}}{2g (2g-2)!}\, \mathrm{Li}_{3-2g} (Q) \,  \\
&=\frac{1}{\lambda^2} \mathrm{Li}_{3}(Q)  +\frac{B_2}{2}\mathrm{Li}_1(Q)+ \Phi^{con}(\check{\lambda},t)\,, \quad \check{\lambda}=\frac{\lambda}{2\pi}\,, \quad Q=e^{2\pi \I t}\,. 
\end{align*} 

In \cite{alim2020difference} following methods of \cite{Iwaki2}, it was proven that this series satisfies the following difference equation:
\begin{equation}\label{eq:diffeq}
\widetilde{F}^{con}\left(\lambda,t+\check{\lambda}\right) + \widetilde{F}^{con}\left(\lambda,t-\check{\lambda}\right) - 2 \widetilde{F}^{con}\left(\lambda,t\right)=-\textrm{Li}_1(Q) \,, \quad \check{\lambda}=\frac{\lambda}{2\pi}\,.\\
\end{equation}

A solution in terms of the triple sine function of this difference equation was found \cite{Alim:2021lld}:
\begin{equation}
\begin{split}
\widetilde{F}^{con}_{\text{np}}(\lambda,t) := \log \left( \mathcal{S}_3(t | \check{\lambda},1)\right)\,,
\end{split}
\end{equation}
where 
$$ 
    \mathcal{S}_3(z\, | \, \omega_1,\omega_2) := \exp\left(\frac{\pi \I}{6} \cdot B_{3,3}(z+\omega_1\,|\,\omega_1,\omega_1,\omega_2)\right) \cdot \sin_3(z+\omega_1\, |\, \omega_1,\omega_1,\omega_2),$$
and the relevant definitions for the Bernoulli polynomials and the triple sine function can be found in \cite{Alim:2021lld}. The non-perturbative content of this solution was analyzed in \cite{Alim:2021ukq} and in \cite{Alim:2021mhp}. In \cite{2023Alim} it was furthermore shown that $\widetilde{F}^{con}_{\text{np}}(\lambda,t)$ satisfies a further difference equation given by:
\begin{align}\label{eq:diffeq2}
    \widetilde{F}_{\text{np}}^{con}(\lambda,t+1)-\widetilde{F}_{\text{np}}^{con}(\lambda,t)=\frac{1}{2\pi \I} \frac{\partial}{\partial \check{\lambda}} \left(  \check{\lambda} \, \textrm{Li}_2(e^{2\pi \I t/\check{\lambda}})\right)\,.
\end{align}

$\widetilde{F}^{con}_{\text{np}}(\lambda,t)$  was identified in \cite{Alim:2021mhp} as the Borel summation of the asymptotic series along a distinguished ray on the real axis in the Borel plane, furthermore the inhomogeneous piece of the difference equation  \eqref{eq:diffeq2} corresponds to the Stokes jumps of the Borel summation of the asymptotic series. $\widetilde{F}^{con}_{\text{np}}(\lambda,t)$ could furthermore be decomposed in the following way, see \cite{Alim:2021ukq,Alim:2021mhp}:
\begin{equation}\label{Fnp-decomp}
\widetilde{F}^{con}_{\text{np}}(\lambda,t)= \sum_{k=1}^\infty\frac{e^{2\pi \I k t}}{k\big(2\sin\big(\frac{\lambda k}{2}\big)\big)^2}  -\frac{1}{2\pi \I }\partial_{\lambda}\left(\lambda \sum_{m=1}^{\infty}\frac{Q'^m}{m^2(1-q'^m)}\right)\,,
\end{equation}

where  $Q'=\exp(2\pi i t/\check{\lambda})\,, \quad q'=\exp(2\pi i/\check{\lambda})$\,.
Based on the results of \cite{Alim:2021mhp}, a product structure was given in \cite{alim2024nonperturbativetopologicalstringsresurgence}

\begin{align}
Z^{con}_{np}(\lambda,t)&= \prod_{m=0}^{\infty} \left(1-Q\, q^{m+1}\right)^{m+1}  \cdot \exp\left(-\frac{1}{2\pi i} \textrm{Li}_2\left(Q' q'^m\right)\right) \cdot \left(1-Q' q'^m\right)^{-\frac{1}{\check{\lambda}}\left( t+m\right)}\,
\end{align}
where: $Q=\exp(2\pi i t),\quad q=e^{\ii\lambda}\,, \quad Q'=\exp(2\pi i t/\check{\lambda})\,, \quad q'=\exp(2\pi i/\check{\lambda})$\,.

\subsection{Nekrasov--Shatashvili limit}

A refinement of topological strings can be introduced, see \cite{Iqbal_2009} and references therein, which in the case of the resolved conifold gives:

\begin{equation}
F^{\textrm{ref}}_{\text{GV}}(\epsilon_1,\epsilon_2,t)= -\sum_{k=1}^{\infty} \frac{Q^k}{k(2\sin \left(k\epsilon_1/2\right))(2\sin \left(k\epsilon_2/2\right))}\,\,.
\end{equation}

\begin{itemize}

\item In the limit $\lim_{\epsilon_1=-\epsilon_2 = \lambda} F^{\textrm{ref}}_{\text{GV}}(\epsilon_1,\epsilon_2,t)$,  the refined free energy reduces to $F_{GV}(\lambda,t)$. 

\item The Nekrasov--Shatashvili (NS) limit \cite{NEKRASOV_2010} of the refined topological string is given by

 \begin{align}\label{FNSgenfunction}
   F^{\text{NS}}(\epsilon,t)&:= \lim_{\epsilon_2\rightarrow 0} \epsilon_2\cdot F^{\textrm{ref}}_{\textrm{GV}}(\epsilon_1,\epsilon_2,t)|_{\epsilon_1=\epsilon}\, =- \frac{1}{2} \sum_{k=1}^{\infty} \frac{Q^k}{k^2 \sin(k\epsilon/2)}
   \end{align}

\end{itemize}

The analytic function obtained from the Borel summation \eqref{Fnp-decomp}
in terms of the variables 
\begin{equation} \label{eq:dualvariables}
\lambda_D= \frac{4\pi^2}{\lambda},\, t_D=\frac{2\pi t}{\lambda} \quad \textrm{and} \quad \check{\lambda}_D= \frac{\lambda_D}{2\pi}=\frac{1}{\check{\lambda}}
\end{equation} 
can thus be cast in the following form, see \cite{Alim:2021ukq,Alim:2021mhp}:

 \begin{equation}\label{GVNS}
\widetilde{F}^{con}_{\text{np}}(\lambda,t)= F_{\text{GV}}(\lambda,t) + \frac{1}{2\pi} \frac{\partial}{\partial \lambda} \left(\lambda \, F_{\text{NS}} \left(\lambda_D,t_D - \frac{1}{2}\check{\lambda}_D\right)\right)\,,
\end{equation}
we use 
\begin{equation}
    \lambda \partial_{\lambda}= -t_D \partial_{t_{D}} - \lambda_D \partial_{\lambda_D}
\end{equation}

to write the second part of \eqref{Fnp-decomp} and \eqref{GVNS} as a sum of three contributions:

\begin{equation}
\begin{split}\label{derivativeFNS}
\frac{1}{2\pi} \frac{\partial}{\partial \lambda} \left(\lambda \, F_{\text{NS}} \left(\lambda_D,t_D - \frac{1}{2}\check{\lambda}_D\right)\right)
&= \frac{1}{2\pi} F_{\text{NS}} \left(\lambda_D,t_D - \frac{1}{2} \check{\lambda}_D\right) \\
&+ \frac{1}{2\pi} \lambda_D F_{GV} (\lambda_D,t_D) \\ &-\frac{1}{2\pi} t_D \partial_{t_D}F_{\text{NS}} \left(\lambda_D,t_D - \frac{1}{2}\check{\lambda}_D\right)\,.
\end{split}
\end{equation}

The enumerative geometric content of \eqref{derivativeFNS} is therefore a combination of three constituents. The first term on the right-hand side is the Nekrasov--Shatashvili free energy, which we will show to be (related to) the generating function of relative Gromov--Witten invariants of $\mathbb P^1 $ relative to the point $\infty$. The second constituent is the usual closed Gromov--Witten theory, in the dual variables $\lambda_D$ and $t_D$ with an additional prefactor of $\lambda_D$. Lastly, the third term of the expression can be identified with the quantum $\mathsf B$-period associated to the resolved conifold using quantum special geometry relations \cite{Aganagic_2012}, with an additional prefactor $t_D$ in front.

The asymptotic expansion near $\lambda_D=0$ of the three constituents is given by:

\begin{equation}
\begin{split}\label{derivativeFNSasympt}
\frac{1}{2\pi} \frac{\partial}{\partial \lambda} \left(\lambda \, F_{\text{NS}} \left(\lambda_D,t_D - \frac{1}{2}\check{\lambda}_D\right)\right)
&\sim \frac{1}{2\pi i} \sum_{m=0}^{\infty} (i \lambda_D)^{m-1} \frac{B_m}{m!} \operatorname{Li}_{3-m} (Q')
\\
&+ \frac{1}{2\pi}\left( \frac{1}{\lambda_D} \mathrm{Li}_{3}(Q')+ \sum_{g=1}^{\infty} \lambda_D^{2g-1} \frac{(-1)^{g-1}B_{2g}}{2g (2g-2)!} \mathrm{Li}_{3-2g} (Q')\right)
\\ &-i  t_D  \sum_{m=0}^{\infty} (i \lambda_D)^{m-1} \frac{B_m}{m!} \operatorname{Li}_{2-m} (Q')\,.
\end{split}
\end{equation}
where we have used
\begin{equation}
    \frac{1}{2\pi i} \partial_{t_D} \operatorname{Li}_{m}(Q')= Q' \partial_{Q'} \operatorname{Li}_{m}(Q')= \operatorname{Li}_{m-1}(Q')\,.
\end{equation}

\section{Relative Gromov--Witten invariants of the conifold}\label{sec:relative gw invariants}
Our main goal is to show a duality naturally arising in the context of resurgence  between two a priori very different curve counting theories associated to the resolved conifold in different regimes of $\lambda$ of the analytic object $\widetilde{F}^{con}_{\text{np}}(\lambda,t)$: closed- and relative GW theory. In order to establish this duality, we will compute relative curve counts for a geometry associated to the resolved conifold first. This will then allow us to compare the results to what we expect from the non-perturbative object obtained via resurgence on the closed GW generating function.\par
A reason for expecting an interplay of different curve counting theories comes from physics. Given a CY threefold one expects a relation between open GW counts and closed GW counts of a dual theory, see for instance \cite[Section 6]{Alim:2021mhp}. It is precisely this understanding which naturally leads to questions about relative GW theory which can be seen as an algebraic approach to open GW theory, a perspective first taken in \cite{Li_2006}.\par
The main objective of this section is to make the reader familiar with the general setup of relative curve counting on $\mathbb P^1$. Furthermore, we will touch on the refined topological string free energy and its relation to relative invariants, based on discussions in \cite{Bousseau_2021} and \cite{brini2024refinedgromovwitteninvariants}.\par
In order to compute the generating function of these invariants, we will give a brief exposition of the work presented in \cite{faber2003relativemapstautologicalclasses, Li_2006} which discusses the moduli space of stable relative maps and its intersection theory. \par
\subsection{Basic properties of \texorpdfstring{$\overline{\mathcal M}_{g,n}$}{Mg,n}}
Let us first introduce the natural classes appearing in the various flavors of GW theory. From now on denote by $\mathcal{M}_{g,n}$ the moduli space of smooth curves of genus $g \in \mathbb Z_{\geq 0}$ and $n \in \mathbb Z_{\geq 0}$ marked points. By $\overline{\mathcal{M}}_{g,n}$ we mean the Deligne-Mumford compactification of $\mathcal M_{g,n}$, see \cite{vakil2003moduli} and references therein. Restrict from now on to the stable range $2g-2+n>0$ unless specified otherwise. It is well known that 
\begin{equation}
    \dim_{\mathbb C} \overline{\mathcal M}_{g,n} = 3g -3 + n,
\end{equation}
and that $\overline{\mathcal M}_{g,n}$ comes with a fundamental cycle
\begin{equation}
    \left[\overline{\mathcal M}_{g,n} \right] \in A_{3g-3+n}\left( \overline{\mathcal M}_{g,n}\right).
\end{equation}\par
\begin{dfn}[$\psi$-classes]\label{psiclasses definition}
    Let $\overline{\mathcal M}_{g,n}$ be the moduli space of Deligne-Mumford stable curves of genus $g \in \mathbb Z_{\geq 0}$ and $n \in \mathbb Z_{\geq 0}$ marked points. Consider the $n$ line bundles $\mathbb L_1 , \dots , \mathbb L_n$ on $\overline{\mathcal M}_{g,n}$, where $\mathbb L_i$ has fiber at the point $[C, p_1, \dots , p_n] \in \overline{\mathcal M}_{g,n}$ the one-dimensional complex cotangent space of $C$ at the marked point $p_i$. We will denote the Chern classes of said line bundles by
    \begin{equation}
        \psi_i := \mathsf c_1 \left( \mathbb L_i\right) \in A^1 \left( \overline{\mathcal M}_{g,n}\right).
    \end{equation}
\end{dfn}
\begin{dfn}[$\lambda$-classes]\label{lambdaclasses definition}
    Let $\overline{\mathcal M}_{g,n}$ be the moduli space of Deligne-Mumford stable curves of genus $g$ with $n$ marked points. Given the universal curve $\pi: \overline{\mathcal C}\to \overline{\mathcal M}_{g,n} $, the Hodge bundle $\mathbb E$ is the rank $g$
    vector bundle on $\overline{\mathcal M}_{g,n}$ whose fiber over a stable curve $[C, p_1, \dots, p_n] \in \overline{\mathcal M}_{g,n}$ is the vector space of global sections of the relative dualizing sheaf $\omega_{\overline{\mathcal C}\vert _C \setminus \overline{\mathcal M}_{g,n}}$. 
    We denote the $i$th Chern class of $\mathbb E$ by
    \begin{equation}
        \lambda_i := \mathsf c_i \left( \mathbb E\right) \in A^i \left( \overline{\mathcal M}_{g,n}\right).
    \end{equation}
\end{dfn}\par
These classes were first considered in \cite{Mumford1983} for the $n=0$ case. Less technically speaking, the $\lambda_i$ class keeps track of degenerations of global holomorphic differentials.

\subsection{Relative GW for conifold}\label{relativeGWsubsection}
As established in \cite{bousseau2020proofntakahashisconjecturemathbbp2e} for the log CY pair $(\mathbb P^2, D) $ for $D$ a smooth elliptic curve, and expanded upon in \cite{brini2024refinedgromovwitteninvariants} to all pairs which are del Pezzo surfaces $S$ with anticanonical divisors $D$, the free energy in the Nekrasov--Shatashvili limit $F_{\mathrm{NS}}\left( \epsilon , t\right)$ for $X =\mathrm{Tot}_S \left( K_S\right)$ is the generating function of the log (i.e. relative) GW invariants of the log Calabi--Yau pair $(S, D)$.

The local $\mathbb P^1$ case is different in nature from local surface cases previously studied, as the only compact components of $X$ will be of complex dimension one, so a divisor is a formal sum of points.

Nonetheless, we find a similar relation between the generating function of log/relative curve counts on $(\mathbb P^1, D)$ and the Nekrasov--Shatashvili free energy on local $\mathbb P^1$ geometries analogous to the one expected from the local surface pair, for instance $(\mathbb P^2, D)$, examined in \cite{bousseau2020proofntakahashisconjecturemathbbp2e, Bousseau_2021} and \cite{brini2024refinedgromovwitteninvariants}. 

As noted in the beginning of this section, the resolved conifold requires special treatment as it is a rank two holomorphic bundle over $\mathbb P^1$. We will now give an exposition of the work presented in \cite{Faber,faber2003relativemapstautologicalclasses, bryan2004curvescalabiyau3foldstopological}.

The relative GW invariants we want to compute are the numbers associated to virtual holomorphic counts with target $\mathbb P^1$ and maximal contact order, i.e. relative with the profile $(d)$ to the point $\infty$. This problem has already been considered in physics in \cite{Ooguri_2000} and later on in algebraic geometry in \cite{Katz_2006, Li_2009} to make sense of open topological string amplitudes. In \cite{Li_2006} it has then been argued that the relative (hence log by \cite{abramovich2013comparisontheoremsgromovwitteninvariants}) GW theory is an algebraic definition of such open curve counts
\begin{equation}
    \mu : (C, \partial C) \to (X, L)\,
\end{equation}
    where $C$ is a once punctured compact Riemann surface and $L$ is some Lagrangian\footnote{$L$ does not have to be Lagrangian necessarily, see the discussion in \cite{cecotti2009bpswallcrossingtopological} as well as \cite{Bousseau_2020}, where an analogous observation in the context of the quantum tropical vertex is made.} in $X$. To make clear why the relative invariants relate to the open counts, consider the following. Take a holomorphic disk in $X$ asymptotic to a non-compact leg with winding profile $\Vec{w}=(d)$. This is precisely an outer toric (Aganagic--Vafa) brane \cite{Aganagic_2002}. Project down to the base of $\pi : X\to \mathbb P^1$. One obtains a degree $d$ map of a disk into $\mathbb P^1$ with a single point mapping to $\infty$ with contact order $(d)$. The relative curves are hence thought of as open disks which get capped off towards the toric boundary at infinity, and instead of getting counts depending on the winding profile one obtains curve counts with specified contact profile. Therefore we compute the relative invariants $N^{\mathbb P^1 \setminus \infty}_{g,d}$ thought of as an algebraic definition of open GW counts.\par
We will now discuss the analogue of Kim's moduli space of stable log maps, the moduli space of stable relative maps into a parameterized $\mathbb P^1$.\par
\begin{dfn}[Stable relative map, {\cite[0.2.2]{faber2003relativemapstautologicalclasses}}]
    Let $g, n \in \mathbb Z_{\geq 0}$ be the genus and number of marked points of a source curve. Furthermore, let $\mu ^1, \dots, \mu ^m$ be $m$ partitions of $d \in \mathbb Z_{\geq 1}$. A stable relative map to a parameterized\footnote{We restrict the discussion to parameterized $\mathbb P^1$, indicated by the dagger superscript.} $\mathbb P^1$ 
    \begin{equation}
        \left( (C; p_1, \dots, p_n; \{Q_1\}, \dots,\{ Q_m\} ) \overset{f}{\longrightarrow} (T; q_1, \dots, q_m) \overset{\epsilon}{\to} \mathbb P^1\right) \in \overline {\mathcal M}^\dagger_{g,n}(\mu ^1, \dots , \mu^m)\,
    \end{equation}
    consists of the following data
    \begin{enumerate}
        \item A projective, connected, nodal curve $C$ of genus $g$ with $n + \sum_{i=1}^m{l (\mu^i)}$ distinct smooth marked points $\{ p_1, \dots, p_n\} \cup \bigcup_{i=1}^m{\{Q_i\}}$, where the set $\{ Q_i\} $ consists of $\vert \{Q_i\} \vert = l(\mu^i)$ many points;
        \item A projective, connected, rational nodal curve $T$ with $m$ smooth markings $q_i$. In other words, the datum of a tree with marked points strictly on the edges;
        \item A structure map $\epsilon : T \to \mathbb P^1$, which is an isomorphism $\epsilon \vert_ P: T \overset{\sim}{\to} \mathbb P^1$  for a unique rational component $\mathbb P^1 \simeq P \subset T$, and contracts all other components $\epsilon: T \setminus P \to \mathrm{pt} \in \mathbb P^1$;
        \end{enumerate}
        such that the following conditions are satisfied:
        \begin{itemize}
        \item All extremal components of $T\setminus P$ are marked by a $q_i$ at least once;
                    \item Matching branchings with neither markings nor contracted components of $C$ are lying over the nodes of the tree $T$;
        \item The map $f$ has the property that over a given smooth marked point $q_i \subset T$, $f$ has profile $\mu^i$. Furthermore, $\{Q_i\}$ is a complete marking of the fiber $f^{-1}(q_i)$ with $l(\mu^i)$ distinct points;
        \item The group of curve automorphisms that respect the marking data commutes with $f$, $\epsilon$, and $\mathsf{id} : \mathbb P^1 \to \mathbb P^1$, and is finite (Deligne-Mumford stability condition).
    \end{itemize}
\end{dfn}
\begin{rem}
    There exist natural maps
    \begin{equation}
        \rho : \overline{\mathcal M}^\dagger _{g,n}(\mu^1, \dots, \mu^m) \to \overline{\mathcal M}_{g, n + \sum_i^m l(\mu^i)}\,
    \end{equation}
    for the stable genera $g \geq 2$. In the rational and elliptic cases, these moduli spaces have expected dimension zero.
\end{rem}

Of special interest to us will be the moduli space of relative stable maps into parameterized $\mathbb P^1$ with only one relative point, the point at $\infty$, equipped with maximal contact order $\mu^1 = (d)$, the one-part partition of the degree $d \in \mathbb Z_{\geq 1}$.

Following the philosophy of \cite{kontsevich1992intersection}, one then reduces the problem of the integral against the virtual fundamental class of the torus fixed locus, which we denote by $\overline{\mathcal M}^\dagger_{g,n} (\mu^1) ^{\mathbb C^\times}$, to a problem about enumerating weighted graphs, such that each vertex contributes factors corresponding to intersection numbers in the moduli space of stable curves and each edge contributes combinatorial (Hurwitz-theoretical) factors, see for example \cite{chiang1999localmirrorsymmetrycalculations} for a discussion on this for closed GW invariants.\par
Based on the in-depth discussion in \cite{bryan2006localgromovwittentheorycurves}, we use virtual torus equivariant localization, developed in \cite{graber1997localizationvirtualclasses}. Thus we have to identify the torus fixed points of $\overline{\mathcal M}_{g,n}^\dagger (\mu^1)$.\par 

\begin{dfn}[{\cite[1.3.3]{faber2003relativemapstautologicalclasses}}]
    Each element of $\overline{\mathcal M}^\dagger_{g,n} (\mu^1, \dots, \mu^m)$ which is fixed by the torus action, lifted from the natural $\mathbb C^\times$ action on $\mathbb P^1$, can be labeled with a graph consisting of the following data. First of all, all marked points, nodal singularities, ramification points and contracted components of $C$ must lie over the fixed point set of the torus action on $\mathbb P^1$. Furthermore, each irreducible component $D \subset C$ which maps dominantly\footnote{We call a component dominant (with respect to a branch map) if it hits a nonempty open set in the target.} onto $\mathbb P^1$ must be a Galois cover with full ramification over the two fixed points. Then a localization graph
    \begin{equation}
        \Gamma = \left(\mathsf V, \mathsf E, \mathsf N, \gamma, \pi, \deg \mathsf e_i, (R^1, \dots, R^m)\right)
    \end{equation}
    consists of the following data:
    \begin{enumerate}
        \item vertices $\mathsf v_i \in \mathsf V$ for each connected component of the fiber $(\epsilon \circ f) ^{-1} (p)$ for $p \in \left\{ [1:0], [0:1]\right\} \subset \mathbb P^1$ a torus fixed point;
        \item $\gamma(\mathsf v_i) \in \mathbb Z_{\geq 0}$ the arithmetic genus of the connected component;
        \item $\pi(\mathsf v_i)$  the fixed point in $\mathbb P^1$ for the connected component associated to $\mathsf v_i$;
        \item $\mathsf e_i \in \mathsf E$ is an edge of $\Gamma$ associated to the $i$th non-contracted irreducible component $D_i \subset C$;
        \item $\deg \mathsf e_i$  the degree of the Galois cover of $(\epsilon \circ f)\vert_{D_i}$ and called the degree of the edge $\mathsf e_i$;
        \item $\mathsf N$ is the set of all markings;
        \item  $R^i$ is a refinement of the partition $\mu^j$ given by
        \begin{enumerate}
            \item a choice of torus fixed point $s^j \in \{ \mathrm{[1:0],[0:1]}\} \in \mathbb P^1$
            \item a distribution of parts of $\mu^j$ to the vertices $\pi^{-1}(s^j)$ satisfying that (for the disconnected case) $\mu^j[i]$ gets distributed to $\Gamma_i$ and the sum of parts distributed to a vertex $\mathsf v_k$ is the sum of degrees $\deg \mathsf e_j$ of edges incident to $\mathsf v_k$.
        \end{enumerate}
    \end{enumerate}
\end{dfn}
\begin{rem}[Torus fixed points are labeled by $\Gamma$, {\cite[Section 5]{bryan2004curvescalabiyau3foldstopological}}]
    The set of localization graphs $\{\Gamma\}$ is isomorphic to the set of torus fixed loci of $\overline{\mathcal M}^\dagger_{g,n}(\mu^1)^{\mathbb C^\times}\subset\overline{\mathcal M}^\dagger_{g,n}(\mu^1)$.
\end{rem}\par
\begin{rem}
    For genus $g=0$, there is only one relative stable map fixed by torus action $\overline{\mathcal M}^\dagger_{0,n}(\mu^1)^{\mathbb C^\times}$, a single point. For $ g>0 $, the relevant fixed component is obtained from the embedding
    \begin{equation}
         \overline {\mathcal M}_{g,1}\hookrightarrow \mathcal M^{\dagger}_{g,n}(\mu^1)^{\mathbb C^{\times}} 
    \end{equation} 
    which sends a stable curve with a single marking $(C; p_1) \in \overline{\mathcal{M}}_{g,1} $ to a relative stable map, following the procedure of \cite[Section 4]{Li_2006}. Essentially, the marking is then replaced by a puncture to which one glues a punctured $\mathbb P^1$.
\end{rem}
It remains to identify the contributing graphs and evaluate the associated integrals over $\overline{\mathcal M}_{g,n}$. After localization on $\overline{\mathcal M}_{g,n}^\dagger (\mu^1)^{\mathbb C^\times}$ we end up with a sum over decorated graphs where each vertex will yield an integral over the moduli space of stable curves with tautological class insertion and each edge will yield a combinatorial factor depending on the contact profile to our divisor. \par
From \cite{faber2003relativemapstautologicalclasses} we know the following about the contributing graphs. 
\begin{rem}\label{graph that contributes}
    The only localization graph $\Gamma$ contributing to the relative map count is the one with two vertices $\mathsf v_1, \mathsf v_2$ connected by one edge $\mathsf e_1$, of edge degree $\deg \mathsf e_1 =d$. This is in line with the fact that we want relative invariants relative to one point.
    \begin{figure}[!ht]
\centering
\resizebox{0.3\textwidth}{!}{%
\begin{tikzpicture}
\tikzstyle{every node}=[font=\normalsize]
\draw  (8.75,9.5) circle (0.25cm) node {\normalsize $g$} ;
\draw (8.75,9.75) -- (8.75,11.75)node[pos=0.5,left, fill=white]{$d$};
\draw  (8.75,12) circle (0.25cm) node {\normalsize $0$} ;
\draw [->, >=Stealth] (9.5,10.75) -- (11.25,10.75);
\node [font=\normalsize] at (11.75,10.75) {$\mathbb{P}^1$};
\end{tikzpicture}
}%
\label{fig:graph mapping into p1}\caption{ Depicted is the only contributing graph $\Gamma$ consisting of two vertices $\mathsf v_1, \mathsf v_2$. The first vertex $\mathsf v_1$ which maps into the fixed point at $[1:0]$ is a genus $g$ vertex. The second vertex $\mathsf v_2$, which maps to $[0:1]$ is a degenerate genus $0$ vertex. The two vertices are connected by a single edge of degree $d$.}
\end{figure}
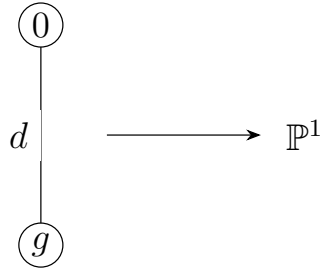\par

\end{rem}\par

With all ingredients at hand we may now compute the relative GW invariants for $\mathbb P^1 \setminus \infty$ with insertion $V$, defined as
\begin{equation}\label{relativeGW computation}
    N^{\mathbb P^1 \setminus \infty}_{g,d} = \int_{\left[ \overline{\mathcal M}^\dagger _{g,n}(\mu^1)\right]^{\mathsf{virt}}}{\mathsf  c_{top}(V)}= \sum_{\Gamma}\frac{1}{\vert \mathrm{Aut}_{\Gamma}\vert} \int_{\left[\overline{\mathcal M}_{\Gamma}\right]^{\mathsf{virt}}}{\frac{\iota^\ast \left(\mathsf c_{top}(V) \right)}{\mathsf e \left( \mathcal N\right)}}
\end{equation}
where from now on (we understand all integrals as $\mathbb C^\times$-equivariant integrals, see \cite{graber1997localizationvirtualclasses}) $\mathfrak t \in \mathbb Z $ is the weight of the torus action on the fibers, $V$ is the obstruction bundle, see \cite[Proposition 3.1]{Li_2006}, and $\overline{\mathcal M}_{\Gamma}$ is the auxiliary moduli space defined in \cite{faber2003relativemapstautologicalclasses}. 

\begin{rem}
Recall that our insertion is $V\coloneqq R^1 \pi_{\ast}\mathsf f^{\ast}\left(\mathcal O_{\mathbb P^1}\oplus \mathcal O_{\mathbb P^1}(-1) \right)$, where the maps $(\pi, \mathsf f)$ come from the universal curve $\mathcal C$ over $\overline{\mathcal M}^{\dagger}_{g,n}\left( (d)\right)$.
    Hence the invariants as defined above \[N^{\mathbb P^1 \setminus \infty }_{g,d} =  \deg_{\left[ \overline{\mathcal M}^{\dagger}_{g,n}\left( (d)\right)\right]^{\mathsf{virt}}}\left(\mathsf c_{top}(V) \cap \left[\overline{\mathcal M}^{\dagger}_{g,n}\left( (d)\right) \right]^{\mathsf{virt}}\right)\] are relative invariants of the quasi-projective threefold $\mathsf{Tot}_{\mathbb P^1}\left( \mathcal O_{\mathbb P^1}\oplus \mathcal O_{\mathbb P^1}(-1) \right)$ relative to the fiber over $\infty \subset \mathbb P^1$. 
\end{rem}

\begin{prop}
    The relative invariants $N^{\mathbb P^1 \setminus \infty}_{g,d}$ as defined above are
    \begin{equation}
    N^{\mathbb P^1 \setminus \infty}_{g,d} =(-1)^{d-1}d^{2g-2} b_g = \begin{cases}
    \frac{(-1)^{d-1}}{d^2},& g=0\\
        (-1)^{d-1} d^{2g-2} \frac{2^{2g-1} -1}{2^{2g-1}} \frac{\vert B_{2g}\vert}{(2g)!}, & g \in \mathbb Z_{\geq 1}
        \end{cases}.
\end{equation}
\end{prop}

\begin{proof}
Let $g>0$ to avoid the unstable case. From Remark \ref{graph that contributes} we know that
\begin{equation}
    \overline{\mathcal M}_{\Gamma} = \overline{\mathcal M}_{g,1} \times \{ \mathrm{pt.}\}
\end{equation}
as the vertex over $\infty$ is degenerate (valence one and genus zero). This justifies using the embedding $\iota$ in the virtual localization.
Then it is shown in \cite{Li_2006} (up to a sign error fixed in \cite{bryan2004curvescalabiyau3foldstopological}) that\footnote{The same term-by-term analysis can be done using the localization graph and the associated Feynman rules described in \cite{bryan2004curvescalabiyau3foldstopological}.}
\begin{equation}
\iota ^\ast \left(\mathsf c_{top}(V)\right) = \frac{(d-1)!}{d^{d-1}} \lambda_g (-\mathfrak t)^{d-1} \left(\mathfrak t^g + \lambda_1 \mathfrak t^{g-1} + \dots + \lambda_g \right),
\end{equation} and the normal bundle $\mathcal N$ coming from the virtual localization satisfies
\begin{equation}
    \mathsf e(\mathcal N)^{-1} =\frac{d^d }{d!}\mathfrak t^{-d} \left(\frac{1}{\mathfrak t/d - \psi_1} \right)\left(\mathfrak t^g + \lambda_1^\vee \mathfrak t^{g-1} + \dots + \lambda_g ^\vee \right),
\end{equation}
where $\mathsf e \left( \mathcal N\right)$ is the Euler class of the normal bundle of $\overline{\mathcal M}_{g,1}$ in $\overline{\mathcal M}^\dagger_{g,n}\left((d) \right)^{\mathbb C^\times}$. Furthermore, $\vert \mathrm{Aut}_\Gamma \vert = d$, which can be computed via the Galois group\footnote{Since the dominant component is a fully ramified degree-$ d $ Galois cover of $ \mathbb P^1 $, its deck transformation group is cyclic of order $ d $.} of the covering.\par

Using Mumford's relation \cite[Eq. 5.4]{Mumford1983}
\begin{equation}
    \mathsf c\left(\mathbb E \right) \mathsf c\left(\mathbb E^\vee \right) =1,
\end{equation}
we can evaluate 
\begin{equation}
    \left(\mathfrak t^g + \lambda_1 \mathfrak t^{g-1} + \dots + \lambda_g \right)\left(\mathfrak t^g + \lambda_1^\vee \mathfrak t^{g-1} + \dots + \lambda_g ^\vee \right)=\mathfrak t^{2g}.
\end{equation}

Therefore we get
\begin{equation}
    \sum_{\Gamma}\frac{1}{\vert \mathrm{Aut}_{\Gamma}\vert} \int_{\left[\overline{\mathcal M}_{\Gamma}\right]^{\mathsf{virt}}}{\frac{\iota^\ast \left(\mathsf c_{top}(V) \right)}{\mathsf e \left( \mathcal N\right)}}=\frac{1}{d} \int_{\left[ \overline{\mathcal M}_{g,1}\right]}(-1)^{d-1}\lambda_g \left(\frac{1}{\mathfrak t/d - \psi_1}\right)\mathfrak t^{2g-1}
\end{equation}
Using that 
\begin{equation}
    \dim_{\mathbb C}\overline{\mathcal M}_{g,1} = 3g-2
\end{equation}
we can simplify the integral, as only a class of degree $3g-2$ in the Chow ring will yield a non-trivial integral. Remember that $\lambda_g \in A^{g}\left(\overline{\mathcal M}_{g,1} \right)$.  The only non-trivial contribution is 
\begin{equation}
\begin{split}
    \frac{1}{d} \int_{\left[ \overline{\mathcal M}_{g,1}\right]}(-1)^{d-1}\lambda_g \left(\frac{1}{\mathfrak t/d - \psi_1}\right)\mathfrak t^{2g-1}= 
    (-1)^{d-1}d^{2g-2} \int_{\left[\overline{\mathcal M}_{g,1}\right]}{\psi_1^{2g-2} \lambda_g}
    \end{split}
\end{equation}
Hence we obtain
\begin{equation}\label{computation simplified for rel GW}
  N^{\mathbb P^1 \setminus \infty}_{g,d} = 
    (-1)^{d-1}d^{2g-2} \int_{\left[\overline{\mathcal M}_{g,1}\right]}{\psi_1^{2g-2} \lambda_g},
\end{equation}
The right-hand side of \eqref{computation simplified for rel GW} 
\begin{equation}\label{intersection numbers bg}
    b_g := \int_{\left[\overline{\mathcal M}_{g,1}\right]}{\psi_1^{2g-2} \lambda_g}, ~g >0
\end{equation}
is understood from the physics perspective, first determined in genus zero (the remaining and unstable case) by \cite{Aspinwall_1993}, yielding $b_0 =1$. We conclude with the result in \cite[Theorem 2]{Faber}
\begin{equation*}
    \sum_{g=0}^\infty b_g x^{2g} = \frac{x/2}{\sin (x/2)} \,
\end{equation*}
valid in all genera $g \in \mathbb Z_{\geq 0}$.
\end{proof}

\begin{rem}[Comment on the sign]
    There is an apparent discrepancy in \cite{Ooguri_2000,Li_2006} and \cite{bryan2004curvescalabiyau3foldstopological} concerning the sign $(-1)^{d-1}$ in the final expression for $N^{\mathbb P^1 \setminus \infty}_{g,d}$. From the physics perspective, one might choose the K\"ahler modulus of the threefold $X$ to come with a sign, i.e. $-Q =\exp(2 \pi i t)$, and then compute the invariants as done in \cite{Aganagic_2012} using the exact WKB method. If our K\"ahler parameter is chosen to be $Q = \exp(2 \pi i t)$ we reproduce the results given in \cite{Ooguri_2000} and \cite{Li_2006}.
\end{rem}
\subsection{Enumerative interpretation of NS limit}
From now on, denote by $F_{\mathrm{NS}}(\lambda_D, T_D)$ the free energy of the refined topological string on $X$ in the Nekrasov--Shatashvili limit \cite{NEKRASOV_2010}, given in \eqref{FNSgenfunction} where we use the notation:
\[T_D \coloneqq t_D - \frac{1}{2}\check{\lambda}_D\,.\]

\par
One can compute $F_{\mathrm{NS}}(\lambda_D, T_D)$ either by solving $q$-difference equations, see \cite{Aganagic_2012}, or thanks to the resurgent analysis discussed above. From the latter perspective, it is predicted  that the NS free energy naturally arises as a constituent of the non-perturbative function $\widetilde{F}^{con}_{\text{np}}(\lambda, t)$, which one can obtain by the procedure in \cite{Alim:2021mhp}.

The previous discussion about relative invariants of $\mathbb P^1\setminus \infty$ then leads us to the following theorem about the enumerative content of $F_{\mathrm{NS}}(\lambda_D,T_D)$ (and hence $\widetilde{F}^{con}_{\text{np}}(\lambda, t)$ via \eqref{Fnp-decomp}) in terms of relative curve counts for the resolved conifold.
\begin{thm}

Fix the notation $Q_T=\exp(2\pi iT)$ and $q'=\exp(i\lambda_D)$. 
Then the Nekrasov--Shatashvili free energy of the resolved conifold satisfies
\begin{equation}
F_{\mathrm{NS}}(\lambda_D,T) = \sum_{g\geq 0} \sum_{d\geq 1} \frac{1}{d} N^{\mathbb{P}^1\setminus\infty}_{g,d} (-Q_T)^d \lambda_D^{2g-1}.
\end{equation}
In particular, evaluating at the shifted dual Kähler parameter $T=T_D$, the relative Gromov--Witten invariants of $\mathbb{P}^1$ with maximal tangency to the divisor $\infty$ give the Nekrasov--Shatashvili contribution appearing in the dual expansion of $\widetilde F^{con}_{\mathrm{np}}(\lambda,t)$.

Equivalently, after the GV-type resummation
one obtains
\begin{equation}
F^{GV}_{\mathrm{NS}}(q',T) = -i \sum_{k\geq 1}
\frac{Q_T^k}{k^2(q'^{k/2}-q'^{-k/2})}.
\end{equation}
Consequently, we obtain for the Poincar\'e polynomials of the moduli space of (degree $d$) Gieseker semistable sheaves on $\mathbb P^1$
\begin{equation}
\Omega_d^{\mathbb{P}^1}(q'^{1/2})=\delta_{d,1}.
\end{equation}
\end{thm}
The above theorem thus establishes that the analytic function $\widetilde{F}^{con}_{\text{np}}(\lambda,t)$ \cite{Alim:2021mhp, 2023Alim}, which we obtain from Borel--Laplace transformation of the closed GW generating function, has strong coupling asymptotics which can be given in terms of relative GW invariants of $\mathbb P^1 \setminus \infty$.

The non-perturbative function has strong-coupling asymptotics recovering open or relative curve counts, and weak-coupling asymptotics recovering closed curve counts. It is interesting that the global object contains information of both curve counting theories. 

We note that both the $\lambda_D$-expansion (GW-type asymptotic series) as well as the $q'$-expansion (GV-type asymptotic series) can also be computed via the exact WKB method, see \cite{Aganagic_2012, Huang_2015, 2023Alim}.
\begin{proof}

Starting from the known \eqref{GVNS} non-perturbative free energy $\widetilde{F}^{con}_{\text{np}}(\lambda,t)$, we extract the Nekrasov--Shatashvili free energy of the topological string on the resolved conifold, and find that it can be written in terms of relative GW invariants. 

Recall from the localization computation that the relative invariants of $\mathbb{P}^1$ with maximal tangency at $\infty$ are
\begin{equation}
    N^{\mathbb{P}^1\setminus\infty}_{g,d} =(-1)^{d-1}d^{2g-2}b_g.
\end{equation}
Here, for $g\geq 1$,
\begin{equation}
b_g\coloneqq \int_{\left[\overline{ \mathcal M}_{g,1} \right]}\psi_1^{2g-2}\lambda_g =
\frac{2^{2g-1}-1}{2^{2g-1}}\frac{|B_{2g}|}{(2g)!},
\end{equation}
and for $g=0$ we have $b_0=1$. The Hodge integrals $b_g$ are encoded by the generating series
\begin{equation}\label{eq:generatingfunc for bg}
    \sum_{g\geq 0}b_g x^{2g} = \frac{x/2}{\sin(x/2)}.
\end{equation}

Now we organize the relative invariants into the  generating function
\begin{equation}
f_{rel}\left(\lambda_D, T \right) \coloneqq\sum_{g\geq 0}
\sum_{d\geq 1}
\frac{1}{d}
N^{\mathbb{P}^1\setminus\infty}_{g,d}
(-Q_T)^d
\lambda_D^{2g-1}\,.
\end{equation}
Substituting the expression for $N^{\mathbb{P}^1\setminus\infty}_{g,d}$ gives
\begin{align}
f_{\mathrm{rel}}(\lambda_D,T)
&=
\sum_{g\geq 0}
\sum_{d\geq 1}
\frac{1}{d}
(-1)^{d-1}d^{2g-2}b_g
(-Q_T)^d
\lambda_D^{2g-1} \\
&=
-\sum_{d\geq 1}
\frac{Q_T^d}{d^3\lambda_D}
\sum_{g\geq 0}b_g(d\lambda_D)^{2g}.
\end{align}
Using the generating series \eqref{eq:generatingfunc for bg} for $b_g$, we obtain
\begin{align}
    f_{rel}(\lambda_D,T)
    &= -\sum_{d\geq 1} \frac{Q_T^d}{d^3\lambda_D} \frac{d\lambda_D/2}{\sin(d\lambda_D/2)} \\
    &= -\frac{1}{2} \sum_{d\geq 1} \frac{Q_T^d}{d^2\sin(d\lambda_D/2)}.
\end{align}
But this is the Nekrasov--Shatashvili free energy of the resolved conifold \eqref{FNSgenfunction}, with K\"ahler parameter $T$ and string coupling the dual coupling $\lambda_D$:
\begin{equation}
    F_{\mathrm{NS}}(\lambda_D,T) = f_{rel}(\lambda_D, T)\, .
\end{equation}
Therefore
\begin{equation}
F_{\mathrm{NS}}(\lambda_D,T)
=
\sum_{g\geq 0}
\sum_{d\geq 1}
\frac{1}{d}
N^{\mathbb{P}^1\setminus\infty}_{g,d}
(-Q_T)^d
\lambda_D^{2g-1}.
\end{equation}

We now rewrite this in GV-type form.
We use 
\begin{equation}
    \frac{1}{2i \sin(k\lambda_D/2)} = \frac{1}{q'^{k/2}-q'^{-k/2}}\,.
\end{equation}
and hence get
\begin{equation}
    F_{\mathrm{NS}}(\lambda_D,T) = -i \sum_{k\geq 1}
\frac{Q_T^k}{k^2(q'^{k/2}-q'^{-k/2})}.
\end{equation}

Comparing this expression with the general GV-type expansion
\begin{equation}
    F^{\mathrm{GV}}_{\mathrm{NS}}(q',T) = -i \sum_{k\geq 1} \sum_{d\geq 1} \frac{1}{k^2} \frac{\Omega_d^{\mathbb{P}^1}(q'^{k/2})} {q'^{k/2}-q'^{-k/2}} Q_T^{dk},
\end{equation}
we see that the only contribution is the degree-one contribution. Therefore we conclude
\begin{equation}
    \Omega_d^{\mathbb{P}^1}(q'^{1/2}) = \delta_{d,1}.
\end{equation}

Finally, in \eqref{Fnp-decomp}, the Nekrasov--Shatashvili free energy is evaluated at the shifted dual Kähler parameter
\begin{equation}
T=T_D \longrightarrow F_{\mathrm{NS}}^{GV} \left(q', T_D \right) = i \sum_{m=1}^{\infty}\frac{Q'^m}{m^2\left(1- q'^m\right)} \,.
\end{equation}
Thus the relative generating function above gives the enumerative content of the Nekrasov--Shatashvili contribution to the dual, or strong-coupling, asymptotic regime of $\widetilde F^{\mathrm{con}}_{\mathrm{np}}(\lambda,t)$.

\end{proof}

\begin{rem}
    A physical interpretation of the result is the following. From the discussion in \cite[Section 6.2]{Aganagic_2012} we argue that the enumeration problem there coincides with relative GW theory of $\mathbb P^1 \setminus \infty$ by the discussion of Section  \ref{relativeGWsubsection} and thus the computation in \cite{Li_2006}, since we are computing Aganagic--Vafa brane \cite{aganagic2000mirrorsymmetrydbranescounting} invariants. The outer toric brane we describe in Section \ref{relativeGWsubsection} seen as a disk mapping to $\mathbb P^1 \setminus \infty$ and a component mapping with specified profile $(d)$ to $\infty \subset \mathbb P^1$ is described by relative GW on parameterized $\mathbb P^1$ with maximal tangency of degree $d \in \mathbb Z_{\geq 1}$ to the point $\infty$. 
\end{rem}

\section{Conclusion and outlook}\label{sec:outlook}
The analytic function $\widetilde{F}^{con}_{\text{np}}(\lambda,t)$ which can be computed from the closed GW generating function consists of two parts: the closed GW generating function and a derivative of the NS free energy after a change of variable. We showed that the NS free energy $F_{NS}(\lambda_D, T_D)$  of the topological string on the conifold is the generating function of the relative GW invariants $N^{\mathbb P^1 \setminus \infty}_{g,d}$. We summarize this interplay between a priori very different curve counting theories in the following schematic diagram:
\[\begin{tikzcd}
	& {\widetilde{F}^{con}_{\text{np}}(\lambda,t)} \\
	{\sim F_{\mathrm{GV}} (\lambda, t)} && {\sim \frac{1}{2\pi}\partial_ \lambda \left(\lambda \, F_{\text{NS}} \left(\lambda_D,T_D\right)\right)}
	\arrow["{\lambda \sim 0}", from=1-2, to=2-1]
	\arrow["{\lambda_D \sim 0}"', from=1-2, to=2-3]
\end{tikzcd}\]
where by tilde we mean the asymptotic expansion in the string coupling variable $\lambda$ and $\lambda_D = \frac{4 \pi^2}{\lambda}$.
From a physics perspective this was expected as described in \cite{Alim:2021mhp}.

It is of interest to study $\widetilde{F}^{\mathrm{CY3}}_{\text{np}}(\lambda,t)$ proposed in \cite{alim2024nonperturbativetopologicalstringsresurgence} for more general local CY threefold geometries, arising from resurgent analysis of the topological string free energy as an asymptotic series in $\lambda$ to obtain and compare refined BPS numbers. For example, one could study the enumerative content of the large-$\lambda$ asymptotic series of $\widetilde{F}^{K_{\mathbb P^2}}_{\text{np}}(\lambda,t)$ and compare to the results computed in  \cite[Section 4.3]{Huang_2015}, to obtain $N^{\mathbb P^2 \setminus E}_{g,d}$ or compare with \cite[Section 6]{bousseau2020scatteringdiagramsstabilityconditions} to obtain $\Omega_d \left(q^{k/2} \right) \in \mathbb Z\left(q^{1/2}\right)$, the GV-type resummed invariants. \par
Another interesting avenue to pursue would be to understand the $q$-difference ($q$-Painlev\'e)  equation discussed in \cite{bonelli2018quantumcurvesqdeformedpainleve} fulfilled by the analytic function $\widetilde{F}^{K_{\mathbb P^1 \times \mathbb P^1}}_{\text{np}}(\lambda,t)$ as well as to find a $q$-difference equation for which $F_{\mathrm{NS}}(\lambda_D, T)$ is a solution. While this was discussed in \cite{2023Alim} for the resolved conifold,  it would be interesting to find similar relations for other (local) CY threefolds. \par
Lastly, it would be interesting to understand to what extent knowledge about finite generation and (quasi-) modularity transfers from the closed GW generating function to the log/relative GW generating function, given that a single function $\widetilde{F}^{\mathrm{CY3}}_{\text{np}}(\lambda,t)$  seems to govern both closed- and relative curve counts and is notably computed from the closed free energy alone.\par


\begin{thebibliography}{CDLOGP98}

\bibitem[Abr08]{AbramovichLectures}
Dan Abramovich.
\newblock {Lectures on {G}romov-{W}itten invariants of orbifolds}.
\newblock In {\em Enumerative invariants in algebraic geometry and string theory}, volume 1947 of {\em Lecture Notes in Math.}, pages 1--48. Springer, Berlin, 2008.

\bibitem[AC10]{chen2010logarithmic}
Dan Abramovich and Qile Chen.
\newblock Logarithmic structure for stable maps relative to simple normal crossing divisor.
\newblock Available at \url{https://www.math.brown.edu/dabramov/LOGGEOM/Relative_To_Simple_Normal_Crossing.pdf}, 2010.

\bibitem[ACD{\etalchar{+}}12]{Aganagic_2012}
Mina Aganagic, Miranda C.~N. Cheng, Robbert Dijkgraaf, Daniel Krefl, and Cumrun Vafa.
\newblock Quantum geometry of refined topological strings.
\newblock {\em Journal of High Energy Physics}, 2012(11), November 2012.

\bibitem[ACG{\etalchar{+}}10]{abramovich2010logarithmic}
Dan Abramovich, Qile Chen, Danny Gillam, Yuhao Huang, Martin Olsson, Matthew Satriano, and Shenghao Sun.
\newblock Logarithmic geometry and moduli.
\newblock {\em arXiv preprint arXiv:1006.5870}, 2010.

\bibitem[ADLM10]{Alim:2008kp}
Murad Alim, Jean Dominique~L{\"a}nge, and Peter Mayr.
\newblock {Global properties of topological string amplitudes and orbifold invariants}.
\newblock {\em Journal of High Energy Physics}, 2010(3):1--30, 2010.

\bibitem[AHT23]{2023Alim}
Murad Alim, Lotte Hollands, and Iván Tulli.
\newblock {Quantum Curves, Resurgence and Exact WKB}.
\newblock {\em Symmetry, Integrability and Geometry: Methods and Applications}, March 2023.

\bibitem[AKV02]{Aganagic_2002}
Mina Aganagic, Albrecht Klemm, and Cumrun Vafa.
\newblock Disk instantons, mirror symmetry and the duality web.
\newblock {\em Zeitschrift für Naturforschung A}, 57(1–2):1–28, February 2002.

\bibitem[AL07]{Alim:2007qj}
Murad Alim and Jean~Dominique L{\"a}nge.
\newblock Polynomial structure of the (open) topological string partition function.
\newblock {\em Journal of High Energy Physics}, 2007(10):045--045, 2007.

\bibitem[Ali23]{alim2020difference}
Murad Alim.
\newblock {Difference equation for the Gromov-Witten potential of the resolved conifold}.
\newblock {\em Journal of Geometry and Physics}, 183:104688, 2023.

\bibitem[Ali24]{alim2024nonperturbativetopologicalstringsresurgence}
Murad Alim.
\newblock {Non-perturbative topological strings from resurgence}.
\newblock {\em arXiv preprint arXiv:2406.17852}, 2024.

\bibitem[Ali25]{Alim:2021ukq}
Murad Alim.
\newblock {Intrinsic non-perturbative topological strings}.
\newblock {\em Adv. Theor. Math. Phys.}, 29(5):1365--1406, 2025.

\bibitem[AM93]{Aspinwall_1993}
Paul~S. Aspinwall and David~R. Morrison.
\newblock Topological field theory and rational curves.
\newblock {\em Communications in Mathematical Physics}, 151(2):245–262, January 1993.

\bibitem[AMW14]{abramovich2013comparisontheoremsgromovwitteninvariants}
Dan Abramovich, Steffen Marcus, and Jonathan Wise.
\newblock {Comparison theorems for Gromov--Witten invariants of smooth pairs and of degenerations}.
\newblock In {\em Annales de l'Institut Fourier}, volume~64, pages 1611--1667, 2014.

\bibitem[AS22]{Alim:2021lld}
Murad Alim and Arpan Saha.
\newblock {On the integrable hierarchy for the resolved conifold}.
\newblock {\em Bulletin of the London Mathematical Society}, 54(5):2014--2031, 2022.

\bibitem[ASTT23]{Alim:2021mhp}
Murad Alim, Arpan Saha, Joerg Teschner, and Iv\'an Tulli.
\newblock {Mathematical Structures of Non-perturbative Topological String Theory: From GW to DT Invariants}.
\newblock {\em Commun. Math. Phys.}, 399(2):1039--1101, 2023.

\bibitem[AV00]{aganagic2000mirrorsymmetrydbranescounting}
Mina Aganagic and Cumrun Vafa.
\newblock {Mirror symmetry, D-branes and counting holomorphic discs}.
\newblock {\em arXiv preprint hep-th/0012041}, 2000.

\bibitem[BBvG24]{bousseau2024stable}
Pierrick Bousseau, Andrea Brini, and Michel van Garrel.
\newblock Stable maps to looijenga pairs.
\newblock {\em Geometry \& Topology}, 28(1):393--496, 2024.

\bibitem[BCOV94]{Bershadsky:1993cx}
Michael Bershadsky, Sergio Cecotti, Hirosi Ooguri, and Cumrun Vafa.
\newblock {Kodaira-Spencer theory of gravity and exact results for quantum string amplitudes}.
\newblock {\em Communications in Mathematical Physics}, 165(2):311--427, 1994.

\bibitem[BFGW21]{Bousseau_2021}
Pierrick Bousseau, Honglu Fan, Shuai Guo, and Longting Wu.
\newblock {Holomorphic anomaly equation for $(\mathbb P^2, E)$ and the Nekrasov-Shatashvili limit of local $\mathbb{P}^2$}.
\newblock {\em Forum of Mathematics, Pi}, 9, 2021.

\bibitem[BGT19]{bonelli2018quantumcurvesqdeformedpainleve}
Giulio Bonelli, Alba Grassi, and Alessandro Tanzini.
\newblock Quantum curves and {$q$}-deformed {P}ainlev\'{e} equations.
\newblock {\em Lett. Math. Phys.}, 109(9):1961--2001, 2019.

\bibitem[Bou20]{Bousseau_2020}
Pierrick Bousseau.
\newblock The quantum tropical vertex.
\newblock {\em Geometry \& Topology}, 24(3):1297–1379, September 2020.

\bibitem[Bou22]{bousseau2020scatteringdiagramsstabilityconditions}
Pierrick Bousseau.
\newblock {Scattering diagrams, stability conditions, and coherent sheaves on $\mathbb{P}^{2}$}.
\newblock {\em Journal of Algebraic Geometry}, 31:593--686, 2022.

\bibitem[Bou23]{bousseau2020proofntakahashisconjecturemathbbp2e}
Pierrick Bousseau.
\newblock {A proof of N. Takahashi’s conjecture for $(\mathbb{P}^2, E)$ and a refined sheaves/Gromov--Witten correspondence}.
\newblock {\em Duke Mathematical Journal}, 172(15):2895--2955, 2023.

\bibitem[BP05]{bryan2004curvescalabiyau3foldstopological}
Jim Bryan and Rahul Pandharipande.
\newblock {Curves in Calabi-Yau 3-folds and Topological Quantum Field Theory}.
\newblock {\em Duke Mathematical Journal}, 126(2), 2005.

\bibitem[BP08]{bryan2006localgromovwittentheorycurves}
Jim Bryan and Rahul Pandharipande.
\newblock {The local Gromov-Witten theory of curves}.
\newblock {\em Journal of the American Mathematical Society}, 21(1):101--136, 2008.

\bibitem[Bri20]{BridgelandCon}
Tom Bridgeland.
\newblock {Riemann--Hilbert problems for the resolved conifold and non-perturbative partition functions}.
\newblock {\em Journal of Differential Geometry}, 115(3):395--435, 2020.

\bibitem[BS24]{brini2024refinedgromovwitteninvariants}
Andrea Brini and Yannik Sch\"uler.
\newblock {Refined Gromov-Witten invariants}.
\newblock {\em arXiv preprint arXiv:2410.00118}, 2024.

\bibitem[CDLOGP98]{Candelas:1990rm}
Philip Candelas, Xenia~C. De~La~Ossa, Paul~S. Green, and Linda Parkes.
\newblock {A Pair of Calabi-Yau manifolds as an exactly soluble superconformal theory}.
\newblock {\em AMS/IP Stud. Adv. Math.}, 9:31--95, 1998.

\bibitem[CK99]{coxkatz}
David~A. Cox and Sheldon Katz.
\newblock {\em Mirror symmetry and algebraic geometry}, volume~68.
\newblock American Mathematical Society Providence, RI, 1999.

\bibitem[CKYZ99]{chiang1999localmirrorsymmetrycalculations}
Ti-Ming Chiang, Albrecht Klemm, Shing-Tung Yau, and Eric Zaslow.
\newblock Local mirror symmetry: Calculations and interpretations.
\newblock {\em Advances in Theoretical and Mathematical Physics}, 3(3):495--565, 1999.

\bibitem[CR02]{ChenRuan}
Weimin Chen and Yongbin Ruan.
\newblock {Orbifold Gromov-Witten theory}.
\newblock {\em Contemporary Mathematics}, pages 25--85, 2002.

\bibitem[CV09]{cecotti2009bpswallcrossingtopological}
Sergio Cecotti and Cumrun Vafa.
\newblock {BPS wall crossing and topological strings}.
\newblock {\em arXiv preprint arXiv:0910.2615}, 2009.

\bibitem[{\'E}ca81]{ecalle1981resurgent}
Jean {\'E}calle.
\newblock {\em Les fonctions r{\'e}surgentes, Vol. I--III}.
\newblock Publications Math{\'e}matiques d'Orsay, Orsay, France, 1981.

\bibitem[FP00]{Faber}
Carel Faber and Rahul Pandharipande.
\newblock {Hodge integrals and Gromov-Witten theory}.
\newblock {\em Inventiones mathematicae}, 139(1):173--199, 2000.

\bibitem[FP05]{faber2003relativemapstautologicalclasses}
Carel Faber and Rahul Pandharipande.
\newblock {Relative maps and tautological classes.}
\newblock {\em Journal of the European Mathematical Society (EMS Publishing)}, 7(1), 2005.

\bibitem[GHK15]{gross2015mirrorsymmetrylogcalabiyau}
Mark Gross, Paul Hacking, and Sean Keel.
\newblock {Mirror symmetry for log Calabi-Yau surfaces I}.
\newblock {\em Publications Math{\'e}matiques de l'IHES}, 122:65--168, 2015.

\bibitem[GHM16]{Grassi:2014zfa}
Alba Grassi, Yasuyuki Hatsuda, and Marcos Mari{\~n}o.
\newblock {Topological Strings from Quantum Mechanics}.
\newblock {\em Annales Henri Poincare}, 17(11):3177--3235, 2016.

\bibitem[GHN23]{Grassi:2022zuk}
Alba Grassi, Qianyu Hao, and Andrew Neitzke.
\newblock {Exponential Networks, WKB and Topological String}.
\newblock {\em SIGMA}, 19:064, 2023.

\bibitem[Giv96]{GiventalMirror}
Alexander~B Givental.
\newblock {Equivariant Gromov-Witten invariants}.
\newblock {\em International Mathematics Research Notices}, 1996(13):613--663, 1996.

\bibitem[GK21]{Garoufalidis:2020pax}
Stavros Garoufalidis and Rinat Kashaev.
\newblock Resurgence of {F}addeev's quantum dilogarithm.
\newblock In {\em Topology and geometry---a collection of essays dedicated to {V}ladimir {G}. {T}uraev}, volume~33 of {\em IRMA Lect. Math. Theor. Phys.}, pages 257--271. Eur. Math. Soc., Z\"urich, [2021] \copyright 2021.

\bibitem[GP99]{graber1997localizationvirtualclasses}
Tom Graber and Rahul Pandharipande.
\newblock {Localization of virtual classes}.
\newblock {\em Inventiones mathematicae}, 135(2):487--518, 1999.

\bibitem[GRZZ26]{gräfnitz2025enumerativegeometryquantumperiods}
Tim Gr{\"a}fnitz, Helge Ruddat, Eric Zaslow, and Benjamin Zhou.
\newblock {Enumerative geometry of quantum periods}.
\newblock {\em Advances in Mathematics}, 499:111063, 2026.

\bibitem[GS06]{gross2006mirror}
Mark Gross and Bernd Siebert.
\newblock {Mirror symmetry via logarithmic degeneration data I}.
\newblock {\em Journal of Differential Geometry}, 72(2):169--338, 2006.

\bibitem[GS10]{gross2010mirror}
Mark Gross and Bernd Siebert.
\newblock {Mirror symmetry via logarithmic degeneration data, II}.
\newblock {\em Journal of Algebraic Geometry}, 19(4):679--780, 2010.

\bibitem[GS13]{gross2013logarithmic}
Mark Gross and Bernd Siebert.
\newblock {Logarithmic Gromov-Witten invariants}.
\newblock {\em Journal of the American Mathematical Society}, 26(2):451--510, 2013.

\bibitem[GS18]{gross2016intrinsic}
Mark Gross and Bernd Siebert.
\newblock {Intrinsic mirror symmetry and punctured Gromov-Witten invariants}.
\newblock In {\em Proceedings of Symposia in Pure Mathematics}, volume~97, 2018.

\bibitem[GV98a]{Gopakumar:1998ii}
Rajesh Gopakumar and Cumrun Vafa.
\newblock {M-Theory and Topological Strings--I}.
\newblock {\em arXiv preprint hep-th/9809187}, 1998.

\bibitem[GV98b]{Gopakumar:1998jq}
Rajesh Gopakumar and Cumrun Vafa.
\newblock {M-Theory and Topological Strings--II}.
\newblock {\em arXiv preprint hep-th/9812127}, 1998.

\bibitem[GV99]{GV}
Rajesh Gopakumar and Cumrun Vafa.
\newblock On the gauge theory/geometry correspondence.
\newblock {\em Advances in Theoretical and Mathematical Physics}, 3(5):1415--1443, 1999.

\bibitem[HK12]{huang2010directintegrationgeneralomega}
Min-Xin Huang and Albrecht Klemm.
\newblock {Direct integration for general $\Omega$ backgrounds}.
\newblock {\em Adv. Theor. Math. Phys.}, 16:805--849, 2012.

\bibitem[HKK{\etalchar{+}}03]{Hori:2003ic}
Kentaro Hori, Sheldon Katz, Albrecht Klemm, Rahul Pandharipande, Richard Thomas, Cumrun Vafa, Ravi Vakil, and Eric Zaslow.
\newblock {\em {Mirror symmetry}}, volume~1 of {\em Clay mathematics monographs}.
\newblock AMS, Providence, USA, 2003.

\bibitem[HKQ08]{Huang:2006hq}
Min-Xin Huang, Albrecht Klemm, and Seth Quackenbush.
\newblock {Topological string theory on compact Calabi--Yau: modularity and boundary conditions}.
\newblock In {\em Homological Mirror Symmetry: New Developments and Perspectives}, pages 1--58. Springer, 2008.

\bibitem[HKR08]{Haghighat:2008gw}
Babak Haghighat, Albrecht Klemm, and Marco Rauch.
\newblock {Integrability of the holomorphic anomaly equations}.
\newblock {\em JHEP}, 10:097, 2008.

\bibitem[HKRS15]{Huang_2015}
Min-Xin Huang, Albrecht Klemm, Jonas Reuter, and Marc Schiereck.
\newblock {Quantum geometry of del Pezzo surfaces in the Nekrasov-Shatashvili limit}.
\newblock {\em Journal of High Energy Physics}, 2015(2), February 2015.

\bibitem[HMMO14]{Hatsuda_2014}
Yasuyuki Hatsuda, Marcos Mariño, Sanefumi Moriyama, and Kazumi Okuyama.
\newblock Non-perturbative effects and the refined topological string.
\newblock {\em Journal of High Energy Physics}, 2014(9), September 2014.

\bibitem[HO15]{Hatsuda:2015owa}
Yasuyuki Hatsuda and Kazumi Okuyama.
\newblock {Resummations and Non-Perturbative Corrections}.
\newblock {\em JHEP}, 09:051, 2015.

\bibitem[IKT19]{Iwaki2}
Kohei Iwaki, Tatsuya Koike, and Yumiko Takei.
\newblock Voros coefficients for the hypergeometric differential equations and {E}ynard-{O}rantin's topological recursion: {P}art {II}: {F}or confluent family of hypergeometric equations.
\newblock {\em J. Integrable Syst.}, 4(1):xyz004, 46, 2019.

\bibitem[IKV09]{Iqbal_2009}
Amer Iqbal, Can Kozçaz, and Cumrun Vafa.
\newblock The refined topological vertex.
\newblock {\em Journal of High Energy Physics}, 2009(10):069–069, October 2009.

\bibitem[IP03]{ionel2003relative}
Eleny-Nicoleta Ionel and Thomas~H. Parker.
\newblock {Relative Gromov-Witten invariants}.
\newblock {\em Annals of Mathematics}, pages 45--96, 2003.

\bibitem[Kat89]{Kato1989LogStructures}
Kazuya Kato.
\newblock {Logarithmic Structures of Fontaine--Illusie}.
\newblock In {\em Algebraic Analysis, Geometry, and Number Theory}, pages 191--224. Johns Hopkins University Press, 1989.

\bibitem[KL06]{Katz_2006}
Sheldon Katz and Chiu-Chu~Melissa Liu.
\newblock {Enumerative geometry of stable maps with Lagrangian boundary conditions and multiple covers of the disc}.
\newblock In {\em The interaction of finite-type and Gromov--Witten invariants (BIRS 2003)}, page 1–47. Mathematical Sciences Publishers, April 2006.

\bibitem[Kon92]{kontsevich1992intersection}
Maxim Kontsevich.
\newblock {Intersection theory on the moduli space of curves and the matrix Airy function}.
\newblock {\em Communications in mathematical physics}, 147(1):1--23, 1992.

\bibitem[Li01]{li2001stable}
Jun Li.
\newblock Stable morphisms to singular schemes and relative stable morphisms.
\newblock {\em Journal of Differential Geometry}, 57(3):509--578, 2001.

\bibitem[LLLZ09]{Li_2009}
Jun Li, Chiu-Chu~Melissa Liu, Kefeng Liu, and Jian Zhou.
\newblock {A mathematical theory of the topological vertex}.
\newblock {\em Geometry \& Topology}, 13(1):527–621, January 2009.

\bibitem[LLY99]{LLY}
Bong~H. Lian, Kefeng Liu, and Shing-Tung Yau.
\newblock Mirror principle. {I} [{MR}1621573 (99e:14062)].
\newblock In {\em Surveys in differential geometry: differential geometry inspired by string theory}, volume~5 of {\em Surv. Differ. Geom.}, pages 405--454. Int. Press, Boston, MA, 1999.

\bibitem[LR01]{li1998symplectic}
An-Min Li and Yongbin Ruan.
\newblock {Symplectic surgery and Gromov-Witten invariants of Calabi-Yau 3-folds}.
\newblock {\em Inventiones mathematicae}, 145(1):151--218, 2001.

\bibitem[LS06]{Li_2006}
Jun Li and Yun~S. Song.
\newblock Open string instantons and relative stable morphisms.
\newblock In {\em The interaction of finite-type and Gromov--Witten invariants (BIRS 2003)}, page 49–72. Mathematical Sciences Publishers, April 2006.

\bibitem[Mar26]{Marino:2024tbx}
Marcos Mari{\~n}o.
\newblock {Les Houches lectures on non-perturbative topological strings}.
\newblock {\em SciPost Physics Lecture Notes}, page 112, 2026.

\bibitem[MNOP06a]{MNOP1}
Davesh Maulik, Nikita~A. Nekrasov, Andrei Okounkov, and Rahul Pandharipande.
\newblock Gromov-{W}itten theory and {D}onaldson-{T}homas theory. {I}.
\newblock {\em Compos. Math.}, 142(5):1263--1285, 2006.

\bibitem[MNOP06b]{MNOP2}
Davesh Maulik, Nikita~A. Nekrasov, Andrei Okounkov, and Rahul Pandharipande.
\newblock Gromov-{W}itten theory and {D}onaldson-{T}homas theory. {II}.
\newblock {\em Compos. Math.}, 142(5):1286--1304, 2006.

\bibitem[MS16]{SauzinLectures}
Claude Mitschi and David Sauzin.
\newblock {\em Divergent series, summability and resurgence. {I}}, volume 2153 of {\em Lecture Notes in Mathematics}.
\newblock Springer, [Cham], 2016.
\newblock Monodromy and resurgence, With a foreword by Jean-Pierre Ramis and a preface by \'{E}ric Delabaere, Mich\`ele Loday-Richaud, Claude Mitschi and David Sauzin.

\bibitem[MS23]{Maulik_2023}
Davesh Maulik and Junliang Shen.
\newblock {Cohomological Chi–independence for moduli of one-dimensional sheaves and moduli of Higgs bundles}.
\newblock {\em Geometry \& Topology}, 27(4):1539–1586, June 2023.

\bibitem[Mum83]{Mumford1983}
David Mumford.
\newblock {\em Towards an Enumerative Geometry of the Moduli Space of Curves}, pages 271--328.
\newblock Birkh{\"a}user Boston, Boston, MA, 1983.

\bibitem[NS10]{NEKRASOV_2010}
Nikita~A. Nekrasov and Samson~L. Shatashvili.
\newblock Quantization of integrable systems and four dimensional gauge theories.
\newblock In {\em XVIth International Congress on Mathematical Physics}. World Scientific, March 2010.

\bibitem[OV00]{Ooguri_2000}
Hirosi Ooguri and Cumrun Vafa.
\newblock Knot invariants and topological strings.
\newblock {\em Nuclear Physics B}, 577(3):419–438, June 2000.

\bibitem[PS10]{Pasquetti:2009jg}
Sara Pasquetti and Ricardo Schiappa.
\newblock {Borel and Stokes Nonperturbative Phenomena in Topological String Theory and c=1 Matrix Models}.
\newblock {\em Annales Henri Poincare}, 11:351--431, 2010.

\bibitem[Tem23]{logintroduction}
Michael Temkin.
\newblock Introduction to logarithmic geometry.
\newblock In {\em New Techniques in Resolution of Singularities}, pages 87--122. Springer, 2023.

\bibitem[Vak03]{vakil2003moduli}
Ravi Vakil.
\newblock The moduli space of curves and its tautological ring.
\newblock {\em Notices of the AMS}, 50(6):647--658, 2003.

\bibitem[vGGR19]{van2019local}
Michel van Garrel, Tom Graber, and Helge Ruddat.
\newblock Local gromov-witten invariants are log invariants.
\newblock {\em Advances in Mathematics}, 350:860--876, 2019.

\bibitem[Wit93]{Witten:1993ed}
Edward Witten.
\newblock {Quantum background independence in string theory}.
\newblock In {\em {Conference on Highlights of Particle and Condensed Matter Physics (SALAMFEST)}}, pages 0257--275, 6 1993.

\bibitem[YY04]{Yamaguchi:2004bt}
Satoshi Yamaguchi and Shing-Tung Yau.
\newblock {Topological string partition functions as polynomials}.
\newblock {\em Journal of High Energy Physics}, 2004(07):047--047, 2004.

\end{thebibliography}

\newcommand{\etalchar}[1]{$^{#1}$}

\end{document}